\documentclass[a4paper,11pt]{article}

\usepackage{amsmath}
\usepackage{amssymb}

\usepackage{url}

\newcommand{\CC}{\mathbb{C}}
\newcommand{\xx}{\mathbf{x}}
\newcommand{\JF}{F_{\xx}}
\newcommand{\HF}{F_{\xx\xx}}
\newcommand{\yy}{\mathbf{y}}
\newcommand{\zz}{\mathbf{z}}
\newcommand{\ee}{\mathbf{e}}
\newcommand{\bb}{\mathbf{b}}
\newcommand{\inter}{\mathrm{int}}
\newcommand{\B}{F_{\xx}^{\cs}(\hat{\xx})}
\newcommand{\JG}{G_{\xx,\boldsymbol{\lambda}_1,\bb_0}}
\newcommand{\JH}{H_{\xx,\boldsymbol{\lambda},\bb}}
\newcommand{\C}{G_{{\xx},{\boldsymbol{\lambda}}_1,\bb_0}^{\cs}(\hat{\xx},\hat{\boldsymbol{\lambda}}_1,\mathbf{0})}
\newcommand{\cs}{\mathbf{c}}
\newcommand{\rs}{\mathbf{k}}
\newcommand{\perF}{\widetilde{F}(\xx,\bb)}

\newcommand{\hh}{\mathbf{h}}

\newcommand{\vv}{\mathbf{v}}
\newcommand{\dete}{\mathrm{det}}
\newcommand{\Span}{\mathrm{Span}}
\newcommand{\rank}{\mathrm{rank}}
\newcommand{\corank}{\mathrm{corank}}

\newcommand{\mux}{\mu}
\newcommand{\lam}{\boldsymbol{\lambda}}
\newcommand{\citep}{\cite}

\newtheorem{example}{EXAMPLE}[section]
\newtheorem{theorem}{Theorem}[section]

\newtheorem{lemma}[theorem]{Lemma}
\newtheorem{algorithm}[theorem]{Algorithm}

\newtheorem{define}[theorem]{Definition}

\newenvironment{proof}{{\bf{Proof. }}}{\hfill $\square$\\}

\title{Verified Error Bounds for Isolated Singular Solutions of Polynomial Systems}
\author{
        Nan Li and Lihong Zhi\\
                KLMM, Academy of Mathematics and Systems Science\\
                CAS, Beijing, 100190, China\\
                linan08@amss.ac.cn, lzhi@mmrc.iss.ac.cn
}

\date{}

\begin{document}
\maketitle \noindent In this paper, we generalize the algorithm described by
Rump and Graillat, as well as our previous work on certifying
breadth-one singular solutions of polynomial systems, to compute
verified and narrow error bounds such that a slightly perturbed
system is guaranteed to possess an isolated singular solution within
the computed bounds. Our new verification method is based on deflation
techniques using smoothing parameters.  We demonstrate the
performance of the algorithm for systems with singular solutions of
multiplicity up to hundreds.

\section{Introduction}
\label{intro}

It is a challenge problem to solve polynomial systems with singular
solutions. In~\citep{Rall66}, Rall studied some convergence
properties of Newton's method for singular solutions, and many
modifications of Newton's method to restore the quadratic
convergence for singular solutions have been proposed
in~\citep{Chen97,DeckerKelley:1980I,DeckerKelley:1980II,DeckerKelley:1982,
Griewank:1980, Griewank85, GriewankOsborne:1981,Ojika:1987,
OWM:1983, Reddien:1978,Reddien:1980,ShenYpma05,YAMAMOTONORIO:1984}.
Recently, some symbolic-numeric methods have also been proposed for
refining approximate isolated singular solutions to high
accuracy~\citep{Corless:1997,DLZ:2009,DZ:2005,GLSY:2005, GLSY:2007,
Lecerf:2002,LVZ06, LVZ:2008, MM:2011,WuZhi:2008, WuZhi:2009}.
In~\citep{LiZhi:2009,LZ:2011}, we described an algorithm based on
the regularized Newton iterations and the computation of
differential conditions satisfied at given approximate singular
solutions to compute isolated singular solutions accurately to the
full machine precision  when its Jacobian matrix has corank one (the
breadth-one case).

%
%
%In \citep{YAMAMOTONORIO:1984}, Yamamoto presented a similar
%deflation method but with some smoothing parameters to derive a
%square system, instead of an over-determined one in \citep{LVZ06}.

Since arbitrary small perturbations of coefficients may transform an
isolated singular solution into a cluster of simple roots  or even
make it disappear, it is more difficult to certify that a polynomial
system or a nonlinear system has a multiple root, if not the entire
computation is performed without any rounding error.

In~\citep{RuGr09},  by introducing a smoothing parameter, Rump and
Graillat described a verification method for computing guaranteed
(real or complex) error bounds such that
 a slightly perturbed system is proved to have a double root within the computed bounds.
In \citep{LZ:2012}, by adding a perturbed univariate polynomial in
one selected variable with some smoothing parameters to one selected
equation of the original system, we generalized the algorithm in
\citep{RuGr09} to compute guaranteed error bounds, such that a
slightly perturbed system is proved to possess an isolated singular
solution whose Jacobian matrix has corank one within the computed
bounds.

%Furthermore, if the original input system has an exact breadth-one
%multiple root near the given approximate singular solution, we can
%refine the multiple root and the multiplicity structure
%simultaneously to arbitrary accuracy.
In \citep{MM:2011}, Mantzaflaris and Mourrain proposed a one-step
deflation method, and by applying a well-chosen symbolic
perturbation, they verified a multiple root of a nearby system with
a given multiplicity structure, which depends on the accuracy of the
given approximate singular solution. The size of the deflated system
is equal to the multiplicity times the size of the original system,
which might be large (e.g. DZ1 and KSS in Table \ref{table1}).

In \citep{Oishi:1999},  based on deflated square systems proposed by
Yamamoto in \citep{YAMAMOTONORIO:1984}, Kanzawa and Oishi presented
a numerical method for proving the existence of ``imperfect singular
solutions'' of nonlinear equations with guaranteed accuracy.  In
\citep{YAMAMOTONORIO:1984}, if the second-order deflation is
applied,
%lzhi rewrite
then smoothing parameters are added not only to the original system
but also to differential systems independently (see
(\ref{Yamaper})). Therefore, one can only prove the existence of an
isolated solution of a slightly perturbed system which satisfies the
first-order differential condition approximately.

In \citep{DK2003MC,KD2003SIAM,KDN2000}, Kearfott  et al. presented
completely different and extremely interesting  methods based on verifying
a nonzero topological degree to certify the existence of singular
zeros of nonlinear systems.

%The techniques for the construction of a parameterized
%deflated system and evaluations of differential functionals are
%similar to those introduced in \citep{LVZ06, LVZ07}.
%
% Standard verification methods
%for nonlinear systems in regular cases are based on
%Theorem \ref{verification}, which requires a square system(variable no.$=$equation no.).

\paragraph{\bf{Main contribution}}
Suppose a polynomial system $F$ and an approximate singular solution are given.
 %, such that  there are no other solutions nearby.
 Stimulated by  our previous work on
certifying breadth-one singular solutions  \citep{LZ:2012}, we show
firstly that the number of deflations used  by Yamamoto to obtain a
regular  system is  bounded by the depth  of the singular solution.
Then we show how to move the independent perturbations in the
first-order differential system (\ref{Yamaper}) appeared in
\cite{YAMAMOTONORIO:1984}  back to the original system. We  prove
that  the modified deflations will  terminate after a finite number
of steps which is bounded by the depth as well, and return a regular
and square augmented system, which can be used to prove the
existence of an isolated singular solution of a slightly perturbed
system exactly, see Theorem \ref{maintheorem} and \ref{breadth}.
% Then we
%modify his approach %in \citep{YAMAMOTONORIO:1984, Oishi:1999}
% by
%adding all smoothing parameters  to the original system $F$, %such
%%that an extended regular system is obtained after finite steps
%(bounded also by the depth of the singular solution)  of deflations.
%Then we apply the standard verification methods (Theorem
%\ref{verification}) to certify an extended regular solution. We
%prove that a slightly perturbed system has an isolated singular
%solution.
 Finally, we present an algorithm for computing verified (real or
complex) error bounds,
  such that a slightly perturbed system is
guaranteed to possess an isolated singular solution within the
computed bounds. The algorithm has been implemented in Maple and
Matlab,  and  narrow error bounds of the order of the relative
rounding error are computed efficiently
for examples given in literature. %The algorithm is also applicable
%to systems of analytic functions.

\paragraph{\bf{Structure of the paper}}
Section \ref{pre} is devoted to recall some notations and well-known
facts. In Section \ref{dfl}, we present a new deflation method by
adding smoothing parameters properly %lzhi was: only
 to the original system, which
will return a regular and square augmented system within a finite
number of steps bounded by the depth. In Section \ref{mrver}, we
propose an algorithm for computing verified (real or complex) error
bounds, such that a slightly perturbed system is guaranteed to
possess an isolated singular solution within the computed bounds.
Some numerical results are given to demonstrate the performance of
our algorithm in Section \ref{exp}.

\section{Preliminaries}
\label{pre}

Let %$R=\mathbb{C}[\xx]$ denote a polynomial ring over the field
%$\mathbb{C}$ of characteristic zero and
$F=\{f_1,\ldots,f_n\}$ be a polynomial system in
$\mathbb{C}[\xx]=\mathbb{C}[x_1, \ldots, x_n]$ and $I\in \CC[\xx]$
be the ideal generated by polynomials in $F$.

\begin{define}
An isolated solution of $F(\xx)=\mathbf{0}$ is a point
$\hat{\xx}\in\mathbb{C}^n$ which satisfies:
\[\mbox{for a small enough } \varepsilon>0: \{\yy\in\mathbb{C}^n:\|\yy-\hat{\xx}\|<\varepsilon\}\cap F^{-1}(\mathbf{0})=\{\hat{\xx}\}.\]
\end{define}

\begin{define}
We call $\hat{\xx}$ a singular solution of $F(\xx)=\mathbf{0}$ if
and only if \begin{align} \rank(\JF(\hat{\xx}))<n,\end{align} where
$\JF(\xx)$ is the Jacobian matrix of $F(\xx)$ with respect to $\xx$.
\end{define}

\begin{define}

Let $Q_{\hat{\xx}}$ be the isolated primary component of the ideal
$I=(f_1, \ldots, f_n)$ whose associate prime is
$m_{\hat{\xx}}=(x_1-\hat{x}_1,\ldots,x_n-\hat{x}_n)$,
%$\hat{\xx}\in\mathbb{C}^n$ an isolated root of $I$,
then the multiplicity $\mu$  of $\hat{\xx}$ is defined as
$\mux=\dim(\CC[\xx]/Q_{\hat{\xx}})$, and the index $\rho$ of $\hat{\xx}$
is defined as the minimal nonnegative integer $\rho$ such that
$m_{\hat{\xx}}^\rho \subseteq Q_{\hat{\xx}}$ \cite{Waerden:1970}.

 %is defined as the dimension of
%the quotient ring $\CC[\xx]/Q_{\hat{\xx}}$.
\end{define}

%\subsection{Local dual space}

 Let
$\mathbf{d}^{\alpha}_{\hat{\xx}}: \CC[\xx] \rightarrow \mathbb{C}$
denote the differential functional defined by
\begin{equation}
\mathbf{d}^{\alpha}_{\hat{\xx}}(g)=\frac{1} {\alpha_1!\cdots
\alpha_n!}\cdot\frac{\partial^{|\alpha|} g}{\partial
x_1^{\alpha_1}\cdots
\partial x_n^{\alpha_n}}(\hat{\xx}),\quad\forall g(\xx)\in \CC[\xx],
\end{equation}
for a point $\hat{\xx}\in \mathbb{C}^n$ and an array $\alpha\in
\mathbb{N}^n$. The normalized differentials have a useful property:
when $\hat{\xx}=\mathbf{0}$, we have
$\mathbf{d}^{\alpha}_{\mathbf{0}}(\xx^{\beta})=1$ if $\alpha=\beta$
or $0$ otherwise.
%We may occasionally write
%$\mathbf{d}^{\alpha}=d_1^{\alpha_1}d_2^{\alpha_2}\cdots
%d_n^{\alpha_n}$ instead of $\mathbf{d}^{\alpha}_{\hat{\xx}}$ for
%simplicity if $\hat{\xx}$ is clear from the context, where
%$d_i^{\alpha_i}=\frac{1}{{\alpha_i}!}\frac{\partial^{\alpha_i}}{\partial
%x_i^{\alpha_i}}$.

\begin{define}\label{localdual}
The local dual space of $I$ at $\hat{\xx}$ is the subspace of
elements of
$\mathfrak{D}_{\hat{\xx}}=\Span_\mathbb{C}\{\mathbf{d}^{\alpha}_{\hat{\xx}},\alpha\in\mathbb{N}^n\}$
that vanish on all the elements of $I$
\begin{equation}
\mathcal{D}_{\hat{\xx}}:=\{\Lambda\in
\mathfrak{D}_{\hat{\xx}}\,\,|\,\, \Lambda(f)=0, ~\forall f\in I\}.
\end{equation}
\end{define}
It is clear  that   $\dim(\mathcal{D}_{\hat{\xx}})=\mux$ and the
maximal degree of an element $\Lambda\in\mathcal{D}_{\hat{\xx}}$ is equal
to the index $\rho-1$, which is also known as   the depth of
$\mathcal{D}_{\hat{\xx}}$.
%\end{define}
%Denote by $\mathcal{D}_{\hat{\xx}}^t$ the subspace of $\mathcal{D}_{\hat{\xx}}$ of the differential degree less than or equal to $t$, for $t\in \mathbb{N}$.
%In \citep{DZ:2005}, they introduced notations of \emph{breadth} and \emph{depth} of $\mathcal{D}_{\hat{\xx}}$, which are computed as \[\dim(\mathcal{D}_{\hat{\xx}}^1)-\dim(\mathcal{D}_{\hat{\xx}}^0)\mbox{ and }
%\min\{t\in\mathbb{N}:\dim(\mathcal{D}_{\hat{\xx}}^{t+1})-\dim(\mathcal{D}_{\hat{\xx}}^t)=0\}\]
%respectively.
%Note that $\hat{\xx}$ is an isolated singular solution of $F(\xx)=\mathbf{0}$ $\Longleftrightarrow$ $1<\dim(\mathcal{D}_{\hat{\xx}})<\infty$.

%\subsection{Deflation method}

A singular solution $\hat{\xx}$ of a square system
$F(\xx)=\mathbf{0}$ satisfies equations
\begin{align}\label{basicdfl}\left\{
\begin{array}{r}
    F(\xx)=\mathbf{0}, \\
    \dete(\JF(\xx))=0.
  \end{array}
\right.
\end{align}
The above augmented system forms the basic idea for the
deflation method \citep{Ojika:1987,Ojika:1988,OWM:1983}.
%since $\hat{\xx}$ has a lower multiplicity as a solution of (\ref{basicdfl}).
But the determinant is usually of high degree, so it is
numerically unstable to evaluate the determinant of the Jacobian
matrix.

In \citep{LVZ06}, Leykin et al. modified  (\ref{basicdfl}) by adding new
variables and equations.
%instead of $\dete(A(\xx))=0$.
Let $r=\rank(\JF(\hat{\xx}))$, then there exists a unique vector
$\hat{\boldsymbol{\lambda}}=$
$(\hat{\lambda}_1,\hat{\lambda}_2\ldots,\hat{\lambda}_{r+1})^T$ such
that $(\hat{\xx},\hat{\boldsymbol{\lambda}})$ is an isolated
solution of
\begin{align}\label{modifieddfl}\left\{
\begin{array}{r}
    F(\xx)=\mathbf{0}, \\
    \JF(\xx)B\boldsymbol{\lambda}=\mathbf{0}, \\
    \hh^T\boldsymbol{\lambda}=1,
  \end{array}
\right.
\end{align}
where $B\in\mathbb{C}^{n\times(r+1)}$ is a random matrix,
$\hh\in\mathbb{C}^{r+1}$ is a random vector and
$\boldsymbol{\lambda}$ is a vector consisting of $r+1$ extra
variables $\lambda_1,\lambda_2\ldots,\lambda_{r+1}$. If
$(\hat{\xx},\hat{\boldsymbol{\lambda}})$ is still a singular
solution of (\ref{modifieddfl}), the deflation is repeated.
Furthermore, they proved that the number of deflations needed to
derive a regular root of an augmented system is strictly less than
the multiplicity of $\hat{\xx}$.  Dayton and Zeng showed that the
depth of $\mathcal{D}_{\hat{\xx}}$ is a tighter bound for the number
of deflations \citep{DZ:2005}.

%\subsection{Verification method}
Let  $\mathbb{IR}$  be the set of real intervals, and
$\mathbb{IR}^n$ and $\mathbb{IR}^{n \times n}$ be  the set of real
interval vectors and real interval matrices, respectively. Standard
verification methods for nonlinear systems are based on the
following theorem \citep{Krawczyk:1969, Moore:1977, Rump:1983}.

%\begin{theorem}\label{verification}
%Let $F(\xx):\mathbb{R}^n\rightarrow\mathbb{R}^n$ be a polynomial
%system, and $\tilde{\xx}\in\mathbb{R}^n$. Given an interval domain $X\in\mathbb{IR}^n$ with
%$\tilde{\xx}\in X$, and an interval matrix $M\in\mathbb{IR}^{n\times
%n}$ satisfies $\nabla f_i(X)\subseteq M_{i,:}$, for $i=1,\ldots,n$.
%Denote by $I$ the $n\times n$ identity matrix and assume
%\begin{equation*}
%-J_F(\tilde{\xx})F(\tilde{\xx})+(I-J_F(\tilde{\xx})M)X\subseteq
%int(X).
%\end{equation*}
%Then there is a unique $\hat{\xx}\in X$ with $F(\hat{\xx})=0$.
%Moreover, every matrix $\tilde{M}\in M$ is nonsingular. In
%particular, the Jacobian matrix $J_F(\hat{\xx})$ is nonsingular.
%\end{theorem}

\begin{theorem}\label{verification}
Let $F(\xx):\mathbb{R}^n\rightarrow\mathbb{R}^n$ be a polynomial
system, and $\tilde{\xx}\in\mathbb{R}^n$. Given $\mathbf{X}\in\mathbb{IR}^n$
with $\mathbf{0}\in \mathbf{X}$ and $M\in\mathbb{IR}^{n\times n}$ satisfies
$\nabla f_i(\tilde{\xx}+\mathbf{X})\subseteq M_{i,:}$, for $i=1,\ldots,n$.
Denote by $I$ the $n\times n$ identity matrix and assume
\begin{equation}
-\JF^{-1}(\tilde{\xx})F(\tilde{\xx})+(I-\JF^{-1}(\tilde{\xx})M)\mathbf{X}\subseteq
\inter(\mathbf{X}).
\end{equation}
Then there is a unique $\hat{\xx}\in \mathbf{X}$ with $F(\hat{\xx})=0$.
Moreover, every matrix $\tilde{M}\in M$ is nonsingular. In
particular, the Jacobian matrix $\JF(\hat{\xx})$ is nonsingular.
\end{theorem}
Naturally the non-singularity of the Jacobian matrix
$\JF(\hat{\xx})$ restricts the application of Theorem
\ref{verification} to regular solutions of square systems. Notice
that Theorem \ref{verification} is valid mutatis mutandis over
complex numbers as well. Next we will use this theorem to derive a
verification method to prove the existence of an isolated singular
solution of a slightly perturbed system.
\section{A Square and Regular  Augmented System  }\label{dfl}

Let a polynomial system $F=\{f_1,\ldots,f_n\}\in \CC[\xx]$ be given
and $\hat{\xx}=(\hat{x}_1,\ldots,\hat{x}_n)$ is an isolated singular
solution satisfying  $F(\hat \xx)=\mathbf{0}$.

The augmented systems
(\ref{basicdfl}) and (\ref{modifieddfl}) % obtained in
%\cite{OWM:1983,Ojika:1987,Ojika:1988,LVZ06,LVZ:2008}
 have been used
to restore the quadratic convergence of Newton's method. But notice that these extended systems are always over-determined, which are not applicable by Theorem
\ref{verification}.
%to prove the existence of the singular solution of $F$.
Hence, a natural thought of modifications is, whether we could add several smoothing parameters to derive a square system with a nonsingular Jacobian matrix.

In \cite{YAMAMOTONORIO:1984},  by introducing smoothing parameters,
Yamamoto derived square deflated systems. These systems were used
successfully  by Kanzawa and Oishi in \cite{Oishi:1999} to certify
the existence of \emph{``imperfect singular solutions''} of
polynomial systems. However, for isolated singular solutions with
high singularities, the smoothing parameters are added not only to
the original system but also to differential
systems independently (see  (\ref{Yamaper})). %Let us assume that the
%second-order deflated system (\ref{Yamaper}) is applicable by
%Theorem \ref{verification}, and yields inclusions for
%%The certified solution
%$\hat \xx$, $\hat{\boldsymbol{\lambda}}$, $\hat{\bb}_0$ and $\hat{\bb}_1$.
%%of the deflated system $H$
%Then the perturbed system
%$\widetilde{F}(\xx,\hat{\bb}_0)=F(\xx)-I_{\rs}\hat \bb_0=\mathbf{0}$ is proved to
%have a solution at $\hat{\xx}$
%with the property
%%which satisfies
%%the first order differential condition only approximately, i.e.,
%$ \widetilde{F}_{\xx} (\hat \xx,\hat{\bb}_0) \hat \vv_1 =\JF(\hat \xx)\hat \vv_1=I_{\rs'}\hat \bb_1$,
%which is not guaranteed to be $\mathbf{0}$.
Therefore, according to (\ref{imper}), one can only prove the existence of an isolated solution
of a slightly perturbed system which satisfies the first-order
differential condition approximately.
%Therefore, according to (\ref{imper}), one can only show
%approximately the existence of an isolated singular solution of a
%slightly perturbed system  using  higher-order deflations defined in
%\cite{YAMAMOTONORIO:1984}.

In the following, we rewrite the deflation techniques in
\cite{YAMAMOTONORIO:1984} in our setting,  and prove that the number
of deflations needed to obtain a regular system  is bounded by the
depth of $\mathcal{D}_{\hat{\xx}}$, see Theorem \ref{terminate}.
Then we show how to lift the independent perturbations in the
first-order differential system appeared in (\ref{Yamaper})  back to
the original system. We prove that the modified deflations will %lzhi remove: also
terminate after a finite number of steps  bounded by the depth of
$\mathcal{D}_{\hat{\xx}}$ as well, and return a regular and square
augmented system, which can be used to verify the existence of an
isolated singular solution of a slightly perturbed system exactly,
see Theorem \ref{maintheorem} and {\ref{breadth}.
%
%
%then our modification
%% of the deflation method in \cite{YAMAMOTONORIO:1984}
%by adding smoothing parameters only to
%the original system.
%%The perturbations to the deflated systems are
%%determined  by the perturbations added to the original system $F$.
%It is interesting to notice that smoothing parameters will be added dependently to differential systems
%under our modification.

\subsection{The first-order deflation}

Let $\hat{\xx}\in\mathbb{C}^n$ be an isolated singular solution of
$F(\xx)=\mathbf{0}$, and
\begin{align}
\rank(\JF(\hat{\xx}))=n-d,  ~(1<d\leq n).
\end{align}
Let $\cs=\{c_1,c_2,\ldots,c_d\}~(1\leq c_1\leq c_2\leq \ldots\leq
c_d\leq n)$ and $\B$ be obtained from $\JF(\hat{\xx})$ by deleting
its $c_1,c_2,\ldots,c_d$-th columns which  satisfies
\begin{equation}\label{col}
 \rank(\B)=n-d. %=\rank(\JF(\hat{\xx})).
 \end{equation}
There exists a positive-integer set $\rs=\{k_1,k_2,\ldots,k_d\}$
such that
 \begin{equation}\label{row}
 \rank(\B,I_{\rs})=n, %\ee_{k_1},\ee_{k_2},\ldots,\ee_{k_d})=n,
 \end{equation}
 where %$I_{\rs}$ consists of the  unit vectors
%$\ee_{k_1},\ee_{k_2},\ldots,\ee_{k_d}$, i.e.,
\begin{equation}
I_{\rs}=(\ee_{k_1},\ee_{k_2},\ldots,\ee_{k_d}), \end{equation} and
$\ee_{k_i}$ is the $k_i$-th unit vector of dimension $n$.

Similar to the augmented system (2.34)  in
\cite{YAMAMOTONORIO:1984}, we
 introduce $d$ smoothing parameters
$\bb_0=(b_1,b_2,\ldots,b_{d})^T$ and consider the following square
 system
\begin{align}\label{Yamapara}G(\xx,\boldsymbol{\lambda}_1,\bb_0)=\left\{
\begin{array}{r}
    F(\xx)-\sum_{i=1}^{d}b_i\ee_{k_i}=\mathbf{0}, \\
    \JF(\xx)\mathbf{v}_1=\mathbf{0},
  \end{array}
\right.
\end{align}
where  $\vv_1$ is a vector consisting of $n-d$ extra variables
$\boldsymbol{\lambda}_1=(\lambda_1,\lambda_2,\ldots,\lambda_{n-d})^T$
and its entries at the positions $c_1,c_2,\ldots,c_d$ are fixed to %lzhi add: be
be $1$ rather than random nonzero numbers used in
\cite{YAMAMOTONORIO:1984}.
%random nonzero constants,  with
%out loss of generality, we assume they are all  $1$.
%By (\ref{col}),
According to (\ref{col}), the rank of $\B$ is $n-d$, the
linear system  $\JF(\hat{\xx})\vv_1=\mathbf{0}$ has a unique
solution, denoted by $\hat{\boldsymbol{\lambda}}_1$. Therefore,
$(\hat{\xx},\hat{\boldsymbol{\lambda}}_1,\mathbf{0})$ is an isolated
solution of (\ref{Yamapara}). If
$(\hat{\xx},\hat{\boldsymbol{\lambda}}_1,\mathbf{0})$ is
%a regular
%solution, then we are done. Otherwise,
still a singular solution,
%of the first-order deflated system $G$,
%$(\hat{\xx},\hat{\boldsymbol{\lambda}_1},\mathbf{0})$ is a singular
%solution, then
 as proposed  in \cite{YAMAMOTONORIO:1984}, the deflation process mentioned above is repeated to the first-order deflated system $G$
 and the solution $(\hat{\xx},\hat{\boldsymbol{\lambda}}_1,\mathbf{0})$.

 Note that Yamamoto did not prove explicitly the termination of  the above-mentioned deflation
 process.   Motivated by the results in \cite{LVZ06,DZ:2005},  we show
 below that  the number of deflations needed to derive a regular and square augmented system  is also bounded by
the depth of $\mathcal{D}_{\hat{\xx}}$.

  %how many steps of deflations will be needed to obtain a regular system.
%will the process terminate after a finite number of deflations?
%In fact, (\ref{Yamapara}) can be derived by adding
%$d$ smoothing parameters to an evolution of (\ref{modifieddfl}).

%Since $r=n-d$,
Let $\hh=(\underbrace{0,\ldots,0}_{n-d},1)^T$,
$\boldsymbol{\lambda}=(\lambda_1,\ldots,\lambda_{n-d},\lambda_{n-d+1})^T$
and
\[B=(\hat{\ee}_{1},\ldots,\underset{c_1}{\hat{\ee}_{n-d+1}},\ldots,\underset{c_{d}}{\hat{\ee}_{n-d+1}},\ldots,\hat{\ee}_{n-d})^T\in\mathbb{C}^{n\times(n-d+1)},\]
%\begin{align*}
%B= \left( \begin{array}{ccccc} 1 & \cdots & 0 &\cdots & 0 &\cdots &0\\
%0& \cdots & 0 &\cdots & 0 &\cdots &0\\
%\end{align*}
where $\hat{\ee}_{i}$ is the $i$-th unit vector of dimension
$n-d+1$. %Then we have
%\begin{equation}\label{equivalent}\left\{
%\begin{array}{r}
%    F(\xx)=\mathbf{0}, \\
%    \JF(\xx)B\boldsymbol{\lambda}=\mathbf{0}, \\
%    \hh^T\boldsymbol{\lambda}=1,
%  \end{array}
%\right.\Longleftrightarrow
%\left\{
%\begin{array}{r}
%    F(\xx)=\mathbf{0}, \\
%    \JF(\xx)\vv_1=\mathbf{0},
%  \end{array}
%\right.
%\end{equation}
Then the augmented system (\ref{modifieddfl}) used in \cite{LVZ06} is
equivalent to
\begin{equation}\label{equivalent}
\widetilde{G}(\xx,\boldsymbol{\lambda}_1)= \left\{\begin{array}{r}
    F(\xx)=\mathbf{0}, \\
    \JF(\xx)\vv_1=\mathbf{0},
  \end{array}
\right.
\end{equation}
which has an isolated solution at $(\hat{\xx},\hat{\boldsymbol{\lambda}}_1)$,
and the Jacobian matrix of $\widetilde{G}(\xx,\boldsymbol{\lambda}_1)$
at $(\hat \xx, \hat {\boldsymbol{\lambda}}_1)$ is
\begin{equation}\label{equivalentfirstorder}
\widetilde{G}_{\xx, {\boldsymbol{\lambda}}_1} (\hat \xx,
\hat{\boldsymbol{\lambda}}_1)=\left(
    \begin{array}{cc}
      \JF(\hat{\xx}) & \mathcal{O}_{n,n-d} \\
      \HF(\hat{\xx})\hat{\vv}_1 & \B\\
    \end{array}
  \right),
  \end{equation}
where $\mathcal{O}_{i,j}$ denotes  the $i\times j$ zero matrix and
$\HF(\xx)$ is the Hessian matrix of $F(\xx)$.
%and $\hat{\boldsymbol{\lambda}}=(\hat{\boldsymbol{\lambda}}_1^T,1)^T$.
%
%regarded (\ref{Yamapara}) as the original system $F(\xx)$ and
%repeated the above-mentioned process again to the system
%$G(\xx,\boldsymbol{\lambda}_1,\bb)$.
%Although this process terminated finitely, one can only certify  the existence of
%\emph{``imperfect singular solutions''} of polynomial systems
%\cite{Oishi:1999}.
%since the smoothing parameters are added to not only the original
%polynomial system but also to the   system $\JF(\xx)\vv_1$.
%In order to adding  smoothing parameters only to the original
%system,
On the other hand, the Jacobian matrix of
$G(\xx,\boldsymbol{\lambda}_1,\bb_0)$ computes to
\begin{equation}\label{J1}
\JG(\hat{\xx},\hat{\boldsymbol{\lambda}}_1,\mathbf{0})=\left(
    \begin{array}{ccc}
      \JF(\hat{\xx}) & \mathcal{O}_{n,n-d} &  -I_{\rs} \\
      \HF(\hat{\xx})\hat{\vv}_1 & \B & \mathcal{O}_{n,d}\\
    \end{array}
  \right).
\end{equation}

\begin{lemma} The null spaces of the Jacobian matrices  (\ref{equivalentfirstorder}) and
(\ref{J1}) satisfy
\[\mathrm{Null}\left(\JG(\hat{\xx},\hat{\boldsymbol{\lambda}}_1,\mathbf{0})\right)
=\left\{\left(\begin{array}{c}
                \yy \\
                \mathbf{0}
              \end{array}
\right)\in\CC^{2n}~|~\yy\in\mathrm{Null}\left(\widetilde{G}_{\xx,\boldsymbol{\lambda}_1}(\hat{\xx},\hat{\boldsymbol{\lambda}}_1)
%\left(
%    \begin{array}{cc}
%      \JF(\hat{\xx}) & \mathcal{O}_{n,n-d} \\
%      \HF(\hat{\xx})\hat{\vv}_1 & \B\\
%    \end{array}
%  \right)
  \right)\right\}.\]

\end{lemma}\label{coranks}
\begin{proof}
If  $\yy
\in\mathrm{Null}\left(\widetilde{G}_{\xx,\boldsymbol{\lambda}_1}(\hat{\xx},\hat{\boldsymbol{\lambda}}_1)\right)$
then $\left(\begin{array}{c}
                \yy \\
                \mathbf{0}
              \end{array}
\right) \in
\mathrm{Null}\left(\JG(\hat{\xx},\hat{\boldsymbol{\lambda}}_1,\mathbf{0})\right)$.
Suppose    $\left(\begin{array}{c}
                \yy \\
                \zz
              \end{array}
\right)$ is a null vector of
$\JG(\hat{\xx},\hat{\boldsymbol{\lambda}}_1,\mathbf{0})$. We divide
$\yy$ into $\left(\begin{array}{c}
                \yy_1 \\
                \yy_2
              \end{array}
\right)$ corresponding to the blocks $\JF(\hat \xx)$ and
$\mathcal{O}_{n,n-d}$.  Therefore, we have
\[\JF(\hat \xx) \yy_1 -I_{\rs} \zz=\mathbf{0}.\]
By (\ref{row}), we have
 \[\rank(\B,-I_{\rs})=n.\]
  It is clear that $\zz$ must be a zero
 vector.
%and we
%divide $\yy$ into $(\yy_1, \yy_2, \yy_3)$ which corresponding to the
%blocks $\B$, the $c_1, c_2, \ldots, c_d$-th columns of $\JF$ and
%$\mathcal{O}_{n,n-d}$.
\end{proof}

%
%While the Jacobian matrix of the right system in (\ref{equivalent}) is equal to the matrix which consists of the
%first two columns of (\ref{J1}).
%%\[
%%\left(
%%    \begin{array}{ccc}
%%      \JF(\hat{\xx}) & \mathcal{O}_{n,n-d} &  -I_{\rs} \\
%%      \HF(\hat{\xx})\hat{\vv}_1 & \B & \mathcal{O}_{n,d}\\
%%    \end{array}
%%  \right).\]
%By (\ref{row}), it is clear that
%\[\corank\left(\left(
%    \begin{array}{cc}
%      \JF(\hat{\xx}) & \mathcal{O}_{n,n-d} \\
%      \HF(\hat{\xx})\hat{\vv}_1 & \B\\
%    \end{array}
%  \right)\right)=\corank\left(\JG(\hat{\xx},\hat{\boldsymbol{\lambda}}_1,\mathbf{0})\right).\]
%Furthermore, we have
%\[\mathrm{Null}\left(\JG(\hat{\xx},\hat{\boldsymbol{\lambda}}_1,\mathbf{0})\right)
%=\left\{\left(\begin{array}{c}
%                \yy \\
%                \mathbf{0}
%              \end{array}
%\right)\in\CC^{2n}~|~\yy\in\mathrm{Null}\left(G_{\xx,\boldsymbol{\lambda}_1}(\hat{\xx},\hat{\boldsymbol{\lambda}}_1,\mathbf{0})
%%\left(
%    \begin{array}{cc}
%      \JF(\hat{\xx}) & \mathcal{O}_{n,n-d} \\
%      \HF(\hat{\xx})\hat{\vv}_1 & \B\\
%    \end{array}
%  \right)
  %\right)\right\}.\]
If $(\hat{\xx},\hat{\boldsymbol{\lambda}}_1)$ is still an isolated
singular solution of the deflated system (\ref{equivalent}), as
proposed in \citep{LVZ06}, the deflation process is repeated for
$\widetilde{G}(\xx,\boldsymbol{\lambda}_1)$ and
$(\hat{\xx},\hat{\boldsymbol{\lambda}}_1)$. Then as shown in
 \citep{DZ:2005},
% the first-order deflated system (\ref{equivalent})
% will bring us new differential functions of the first-order in  $\mathcal{D}_{\hat{\xx}}$. If the deflated system (\ref{equivalent})
%is still singular, then we repeat the deflation step.
% In the  $\alpha$-th deflation step, one will obtain  new differential functions of
%order $\alpha$ in  $\mathcal{D}_{\hat{\xx}}$.
if the $s$-th deflated system is singular,
there exists at least one differential functional of the order $s+1$ in $\mathcal{D}_{\hat{\xx}}$.
However, the order of
differential functionals in $\mathcal{D}_{\hat{\xx}}$ is bounded by
its depth which is equal to $\rho-1$. Therefore, after at most
$\rho-1$  steps of deflations, one will obtain a regular deflated
system, i.e.,
 the corank of the Jacobian matrix of the deflated system  will be zero.

%
% just completes a square system from
%(\ref{equivalent}) by adding smoothing parameters but without
%changing any rank-deficient information of the Jacobian matrix.
As a consequence, based on Lemma \ref{coranks}, we claim the finite
termination of Yamamoto's  deflation method.

\begin{theorem}\label{terminate}
The number of Yamamoto's   deflations %applied to the system
 needed to derive a regular solution of a square augmented system
is bounded by the depth of $\mathcal{D}_{\hat{\xx}}$.
\end{theorem}
\begin{proof}
By Lemma \ref{coranks}, we have
\begin{equation}\label{equalcorank}
\corank\left(\widetilde{G}_{\xx, {\boldsymbol{\lambda}}_1} (\hat
\xx,
\hat{\boldsymbol{\lambda}}_1)\right)=\corank\left(\JG(\hat{\xx},\hat{\boldsymbol{\lambda}}_1,
\mathbf{0})\right).
\end{equation}
Therefore, the smoothing parameters we added in the deflated system
(\ref{Yamapara}) do not change  rank-deficient information of the
Jacobian matrix of the deflated system (\ref{equivalent}). %Moreover,
%it is clear that the corank of the Jacobian matrix
%(\ref{equivalentfirstorder}) is less than or equal to  $d$.
%Hence,  the two sequences consisting of  the coranks of the Jacobian
%matrices of two kinds of deflated systems corresponding to
%(\ref{Yamapara}) and ((\ref{equivalent}) are the same.
%If the Jacobian matrix of the deflated system (\ref{Yamapara}) is
%singular, then the Jacobian matrix of the system (\ref{equivalent})
%will also be singular.
 If $\corank\left(\widetilde{G}_{\xx,
{\boldsymbol{\lambda}}_1} (\hat \xx,
\hat{\boldsymbol{\lambda}}_1)\right)=\corank\left(\JG(\hat{\xx},\hat{\boldsymbol{\lambda}}_1,
\mathbf{0})\right)>0,$ then  we repeat the deflation steps to
(\ref{Yamapara}) and (\ref{equivalent}) accordingly. Inductively, we
know that coranks of Jacobian matrices of two different kinds of
deflated systems remain equal at every step. Moreover, we have shown that,
after as most $\rho-1$ steps, the corank of the Jacobian matrix of
the deflated system corresponding to (\ref{equivalent}) will become
zero. Therefore, the deflated system corresponding to
(\ref{Yamapara}) will also become regular after at most $\rho-1$
steps.
\end{proof}

%\begin{proof}
%Inductively, the two Jacobian matrices of the same order deflated systems which are derived by repeating (\ref{Yamapara}) and (\ref{equivalent}), are of the same corank.
%%and the only difference between their null space
%Since the number of deflations (\ref{modifieddfl}) to derive a regular solution of an augmented system
%is less than the depth of
%$\mathcal{D}_{\hat{\xx}}$, to be repeated at most same times, the deflation method (\ref{Yamapara}) will derive a square system with a nonsingular Jacobian matrix.
%\end{proof}

\subsection{The second-order deflation}
%We consider the second-order deflation.
Suppose the Jacobian matrix
$\JG(\hat{\xx},\hat{\boldsymbol{\lambda}}_1,\mathbf{0})$ is
singular, i.e.,
\begin{align}\rank(\JG(\hat{\xx},\hat{\boldsymbol{\lambda}}_1,\mathbf{0}))=2n-d',~(d'\geq
1).\end{align}
 Let $\cs'=\{c_1',c_2',\ldots,c_{d'}'\}$  and
$\rs'=\{k_1',k_2',\ldots,k_{d'}'\}$ be two positive-integer sets
such that
\begin{equation}\label{fullranksecond1}
\rank(G_{{\xx},{\boldsymbol{\lambda}_1},\bb_0}^{\cs'}(\hat{\xx},\hat{\boldsymbol{\lambda}}_1,\mathbf{0}))=2n-d',
%\mbox{ and
%}
\end{equation}
\begin{equation}\label{fullranksecond2}
\rank\left(G_{{\xx},{\boldsymbol{\lambda}_1},\bb_0}^{\cs'}(\hat{\xx},\hat{\boldsymbol{\lambda}}_1,\mathbf{0}),
I_{\rs'+n}\right)=2n,
\end{equation}
 where
$G_{{\xx},{\boldsymbol{\lambda}_1},\bb_0}^{\cs'}(\hat{\xx},\hat{\boldsymbol{\lambda}}_1,\mathbf{0})$
is a matrix obtained from
$\JG(\hat{\xx},\hat{\boldsymbol{\lambda}}_1,\mathbf{0})$ by deleting
its $c_1',c_2',\ldots,c_{d'}'$-th columns, and
\begin{align}
I_{\rs'+n}=\left(\begin{array}{c}
                                        \mathcal{O}_{n,d'} \\
                                        I_{\rs'}
                                      \end{array}
\right),  ~ I_{\rs'}=( %consists of the  unit vectors
\ee_{k_1'},\ee_{k_2'},\ldots,\ee_{k_{d'}'}). \end{align}
%\left[\begin{array}{ccc}
%                                      \mathbf{0}_{n}& \cdots& \mathbf{0}_{n}\\ % \mathcal{O}_{n,d'} \\
%                                       \ee_{k_1'} & \cdots & \ee_{k_{d'}'}  %I_{\rs'}
 %                                     \end{array}
%\right]$ %and $I_{\rs'}$ consists of $\rs'$-th unit vectors,
% where $\mathbf{0}_{n}$
%denotes  the  zero vector of dimension $n$ and  $\ee_{{k_i}'}$ is
%the $k_i'$-th unit vector of dimension $n$. %Then we have

\begin{theorem}\label{2nddfl} Comparing to $\JF(\hat{\xx})$, the corank of $\JG(\hat{\xx},\hat{\boldsymbol{\lambda}}_1,\mathbf{0})$
% $\JG(\hat{\xx},\hat{\boldsymbol{\lambda}_1},\mathbf{0})$
does not increase, i.e.,  $d' \leq d$. Moreover, we can choose
$\cs'$ and $\rs'$ such that  $\cs' \subseteq \cs$,
$\rs'\subseteq\rs$ and satisfy (\ref{fullranksecond1}) and (\ref{fullranksecond2}) respectively.%have the following relations
%\[~\cs'\subseteq\cs\mbox{ and }\rs'\subseteq\rs.\]
\end{theorem}
\begin{proof}
Let
\[\C=\left(
    \begin{array}{ccc}
      \B & \mathcal{O}_{n,n-d} & -I_{\rs} \\
      \star & \B & \mathcal{O}_{n,d} \\
    \end{array}
  \right),\]
  be the matrix obtained from $\JG(\hat{\xx},\hat{\boldsymbol{\lambda}}_1,\mathbf{0})$ by deleting its $c_1,c_2,\ldots,c_d$-th
  columns. By (\ref{col}) and (\ref{row}) we claim that
\begin{equation*}\label{newrank1}
\rank(\C)=2n-d.
\end{equation*}
Hence $d'\leq d$. Besides there exists a positive-integer set
$\cs'\subseteq\cs$ such that the condition (\ref{fullranksecond1})
is satisfied.
%. It is clear that  we can delete  the columns
%$\{j_1',j_2',\ldots,j_{d'}'\}$ among the columns
%$j_1,j_2,\ldots,j_d$ such that the obtained matrix
%$G_{{\xx},{\boldsymbol{\lambda}_1},\bb_0}^{\cs'}(\hat{\xx},\hat{\boldsymbol{\lambda}_1},\mathbf{0})$
%has full column rank. Hence $\cs'\subseteq\cs$.

According to (\ref{row}), %we have
%\rank(\B,-I_{\rs})=
%$\rank(\B,I_{\rs})=n$, then
it is clear that
\begin{equation*}\label{newrank2}
 \rank(\C,I_{\rs+n})=2n,
\end{equation*}
%where $\C$ is obtained from $\JG(\hat{\xx},\hat{\boldsymbol{\lambda}_1},\mathbf{0})$ by deleting its
%$\cs$-th columns,
where $I_{\rs+n}=\left(\begin{array}{c}
                                        \mathcal{O}_{n,d} \\
                                        I_{\rs}
                                      \end{array}
\right)$. Hence we can choose    $\rs'\subseteq\rs$ such that  the
condition (\ref{fullranksecond2}) is satisfied.
\end{proof}

 If $d' \geq 1$, then Yamamoto repeated the first-order deflation to
$G(\xx,\boldsymbol{\lambda}_1,\bb_0)$ defined by
 (\ref{Yamapara}). By Theorem \ref{2nddfl},  we notice that
 Yamamoto's second-order deflation is equivalent to adding
$d'$ new smoothing parameters denoted by $\bb_1$ to  the first-order
differential system $\JF(\xx)\vv_1=\mathbf{0}$, to  derive a square
system
\begin{align}\label{Yamaper}H(\xx,\boldsymbol{\lambda},\bb)=\left\{
\begin{array}{r}
F(\xx)-I_{\rs}\bb_0=\mathbf{0}, \\
    \JF(\xx)\vv_1-I_{\rs'}\bb_1=\mathbf{0},\\
    G_{\xx,\boldsymbol{\lambda}_1,\bb_0}(\xx,\boldsymbol{\lambda}_1,\bb_0)\vv_2=\mathbf{0},
  \end{array}
\right.
\end{align}
where $\vv_2$ is a vector consisting of $2n-d'$ extra variables
$\boldsymbol{\lambda}_2$ and its entries at the positions
$c_1',c_2',\ldots,c_{d'}'$ are all $1$, and %lzhi add: and
$\bb=(\bb_0^T,\bb_1^T)^T$,
$\boldsymbol{\lambda}=(\boldsymbol{\lambda}_1^T,\boldsymbol{\lambda}_2^T)^T$.
Let $\hat{\boldsymbol{\lambda}}_2$ denote the unique solution of the
linear system
$G_{\xx,\boldsymbol{\lambda}_1,\bb_0}(\hat{\xx},\hat{\boldsymbol{\lambda}}_1,\mathbf{0})\vv_2=\mathbf{0}$,
then $(\hat{\xx},\hat{\boldsymbol{\lambda}},\mathbf{0})$ is an
isolated solution of (\ref{Yamaper}).

Suppose  Theorem \ref{verification} is  applicable to the deflated
system $H(\xx,\boldsymbol{\lambda},\bb)$,  and yields inclusions for $\hat \xx$, $\hat{\boldsymbol{\lambda}}$, $\hat{\bb}_0$ and $\hat{\bb}_1$.
%is a certified solution of $H$  satisfying  (\ref{Yamaper}), i.e., we have
Then we have
\begin{align}\label{imper}
\widetilde{F}(\hat \xx)=F(\hat \xx)-I_{\rs}\hat \bb_0=\mathbf{0}
~{\text {and}}~ \widetilde{F}_{\xx} (\hat \xx) \hat \vv_1 =\JF(\hat
\xx)\hat \vv_1=I_{\rs'}\hat \bb_1,
 \end{align}
where smoothing parameters $\hat \bb_1$ might be very small, but are
not guaranteed to be zeros.  Therefore,  one can only prove the
existence of an isolated solution $\hat \xx$ of a perturbed system
$\widetilde{F}(\xx)$, which satisfies the first-order differential
condition approximately.

  In order to  verify  the   existence of an isolated singular
solution of a slightly perturbed system, we should  add the
smoothing parameters $\bb_1$  back to the original system. Let us
consider the modified system:
\begin{align}\label{liper}\widetilde{H}(\xx,\boldsymbol{\lambda},\bb)=\left\{
\begin{array}{r}
F(\xx)-I_{\rs}\bb_0-X_1\bb_1=\mathbf{0}, \\
    \JF(\xx)\vv_1-I_{\rs'}\bb_1=\mathbf{0},\\
    \widetilde{G}_{\xx,\boldsymbol{\lambda}_1,\bb_0}(\xx,\boldsymbol{\lambda}_1,\bb_0,\bb_1)\vv_2=\mathbf{0},
  \end{array}
\right.
\end{align}
where $X_1=(x_{c_1'} \ee_{k_1'}, \ldots, x_{c_{d'}'} \ee_{k_{d'}'})$
and
%
%
%$x_{\cs'(i)}\ee_{\rs'(i)}$, $i=1,\ldots,d'$ and
\begin{align}\widetilde{G}(\xx,\boldsymbol{\lambda}_1,\bb_0,\bb_1)=\left\{
\begin{array}{r}
F(\xx)-I_{\rs}\bb_0-X_1\bb_1=\mathbf{0}, \\
    \JF(\xx)\vv_1-I_{\rs'}\bb_1=\mathbf{0}.
  \end{array}
\right.
\end{align}

\begin{theorem}\label{singularsolution}
Let
\begin{align}\label{perF}
\perF=F(\xx)-I_{\rs}\bb_0-X_1\bb_1,
\end{align}
 then we  have
\begin{align}
\JF(\xx)\vv_1-I_{\rs'}\bb_1=\mathbf{0}\Longleftrightarrow\widetilde{F}_{\xx}(\xx,\bb)\vv_1=\mathbf{0}.
\end{align}
\end{theorem}

\begin{proof}
Let \[\bb_1=(b_{1,1},b_{1,2},\ldots,b_{1,d'})^T\mbox{ and }\vv_1=(
\lambda_1,\cdots,\underset{c_1}1,\cdots,
  \underset{c_{d}} 1,\cdots,\lambda_{n-d})^T,\]
 then
\begin{align*}
  \widetilde{F}_{\xx}(\xx,\bb)\vv_1  &= \JF(\xx)\vv_1-(\mathbf{0},\cdots,\underset{c'_1}{b_{1,1}\ee_{k'_1}},\cdots,
  \underset{c'_{d'}}{b_{1,d'}\ee_{k'_{d'}}},\cdots,\mathbf{0})\, \vv_1\\
  & = \JF(\xx)\vv_1-(\ee_{k_1'},\ldots,\ee_{k_{d'}'})\, \bb_1 ~({\rm since}~
  %\{j_1',j_2',\ldots,j_{d'}'\} \subseteq \{j_1,j_2,\ldots,j_{d}\}) \\
  \cs'\subseteq\cs)\\
  & = \JF(\xx)\vv_1-I_{\rs'}\bb_1
\end{align*}
\end{proof}

%We can see from Theorem \ref{singularsolution}, new smoothing parameters $\bb_1$ are added to the first-order differential system $\JF(\xx)\vv_1=\mathbf{0}$ dependently,
According to Theorem \ref{singularsolution}}, we can rewrite the
system (\ref{liper}) as
\begin{align}\label{liper2}\widetilde{H}(\xx,\boldsymbol{\lambda},\bb)=\left\{
\begin{array}{r}
\perF=\mathbf{0},\\
\widetilde{F}_{\xx}(\xx,\bb)\vv_1=\mathbf{0},\\
    \widetilde{G}_{\xx,\boldsymbol{\lambda}_1,\bb_0}(\xx,\boldsymbol{\lambda}_1,\bb_0,\bb_1)\vv_2=\mathbf{0}.
  \end{array}
\right.
\end{align}
%where $\perF$ is defined by (\ref{perF}).
Therefore, if we can
certify that $(\hat{\xx},\hat{\boldsymbol{\lambda}},\hat{\bb})$ is a
regular solution of the augmented system
$\widetilde{H}(\xx,\boldsymbol{\lambda},\bb)$ based on Theorem
\ref{verification}, then by (\ref{liper2}),
$\hat{\xx}$ is guaranteed to be an isolated  singular solution of
$\widetilde{F}(\xx,\hat{\bb})$.
%the slightly perturbed system

%From our former analysis, we know that the deflation technique (\ref{Yamapara}) just completes a square system from the deflated system (\ref{equivalent})
%by adding smoothing
%parameters without changing any rank-deficient information of the Jacobian matrix. Next we prove that this property also holds for our modified deflation (\ref{liper2})
%compared to (\ref{Yamapara}).

\begin{theorem}\label{samerank}
Jacobian matrices of (\ref{Yamaper}) and (\ref{liper}) share the same null space.
\end{theorem}
\begin{proof}
The Jacobian matrix $\JH(\hat{\xx},\hat{\boldsymbol{\lambda}},\mathbf{0})$ of (\ref{Yamaper}) computes to
\begin{equation}\label{jaco1}
\left(
\begin{array}{cccc}
      \begin{array}{cc}
      \JF(\hat{\xx}) & \mathcal{O}_{n,n-d}  \\
      \HF(\hat{\xx})\hat{\vv}_1 & \B  \\
    \end{array} & \begin{array}{c}
                    -I_{\rs} \\
                    \mathcal{O}_{n,d}
                  \end{array}
     & \mathcal{O}_{2n,2n-d'} & \begin{array}{c}
         \mathcal{O}_{n,d'} \\
         -I_{\rs'}
       \end{array}
      \\
  \star & \begin{array}{c}
            \mathcal{O}_{n,d} \\
            \mathcal{O}_{n,d}
          \end{array}
   & \begin{array}{ccc}
      F_{\xx}^{\cs'}(\hat{\xx}) & \mathcal{O}_{n,n-d} & -I_{\rs} \\
      \star & \B & \mathcal{O}_{n,d} \\
    \end{array} & \begin{array}{c}
            \mathcal{O}_{n,d} \\
            \mathcal{O}_{n,d}
          \end{array}
\end{array}
\right),
\end{equation}
while the Jacobian matrix
$\widetilde{H}_{\xx,\boldsymbol{\lambda},\bb}(\hat{\xx},\hat{\boldsymbol{\lambda}},\mathbf{0})$
of (\ref{liper}) computes to
\begin{equation}\label{jaco2}
\left(
\begin{array}{cccc}
      \begin{array}{cc}
      \JF(\hat{\xx}) & \mathcal{O}_{n,n-d}  \\
      \HF(\hat{\xx})\hat{\vv}_1 & \B  \\
    \end{array} & \begin{array}{c}
                    {-I_{\rs}} \\
                    \mathcal{O}_{n,d}
                  \end{array}
     & \mathcal{O}_{2n,2n-d'} & \begin{array}{c}
         {-\hat{X}} \\
         -I_{\rs'}
       \end{array}
      \\
  \star & \begin{array}{c}
            \mathcal{O}_{n,d} \\
            \mathcal{O}_{n,d}
          \end{array} & \begin{array}{ccc}
      F_{\xx}^{\cs'}(\hat{\xx}) & \mathcal{O}_{n,n-d} & {-I_{\rs}} \\
      \star & \B & \mathcal{O}_{n,d} \\
    \end{array} & \begin{array}{c}
                    {-I_{\rs'}} \\
                    \mathcal{O}_{n,d'}
                  \end{array}
\end{array}
\right),
\end{equation}
where the matrix $\hat{X}$ consists of vectors
$\hat{x}_{\cs'(i)}\ee_{\rs'(i)}$, $i=1,\ldots,d'$.  Since
$\rs'\subseteq\rs$,
%it is clear that the two Jacobian matrices are of the same rank. Moreover, the null space are also the same.
 we can reduce the last column of
the block matrix (\ref{jaco2}) by its third and sixth columns to get
the block matrix (\ref{jaco1}). Therefore, two Jacobian matrices
(\ref{jaco1}) and (\ref{jaco2}) are of the same corank and share the
same null space.
% On the other hand, since
%\[\mathrm{corank}\left(\JG(\hat{\xx},\hat{\boldsymbol{\lambda}}_1,\mathbf{0})\right)=\mathrm{corank}\left(G_{{\xx},{\boldsymbol{\lambda}_1},\bb_0}(\hat{\xx},\hat{\boldsymbol{\lambda}}_1,\mathbf{0}),
%I_{\rs'+n}\right)=d',\]
%so we have
%\[\mathrm{Null}\left(\JG(\hat{\xx},\hat{\boldsymbol{\lambda}}_1,\mathbf{0})\right)=\mathrm{Null}\left(G_{{\xx},{\boldsymbol{\lambda}_1},\bb_0}(\hat{\xx},\hat{\boldsymbol{\lambda}}_1,\mathbf{0}),
%I_{\rs'+n}\right),\]
%which implies
%\[\mathrm{Null}\left(\JH(\hat{\xx},\hat{\boldsymbol{\lambda}},\mathbf{0})\right)\subseteq\mathrm{Null}\left(\widetilde{H}_{\xx,\boldsymbol{\lambda},\bb}(\hat{\xx},\hat{\boldsymbol{\lambda}},\mathbf{0})\right).\]
\end{proof}

Suppose the Jacobian matrix
$\JH(\hat{\xx},\hat{\boldsymbol{\lambda}},\mathbf{0})$ is still
singular, i.e.,
\begin{align}
\rank(\JH(\hat{\xx},\hat{\boldsymbol{\lambda}},\mathbf{0}))=4n-d'',~(d''\geq
1).
\end{align}
 Let $\cs''=\{c_1'',c_2'',\ldots,c_{d''}''\}$  and
$\rs''=\{k_1'',k_2'',\ldots,k_{d''}''\}$ be two positive-integer
sets such that
\begin{equation}\label{fullranksecond3}
\rank(\JH^{\cs''}(\hat{\xx},\hat{\boldsymbol{\lambda}},\mathbf{0}))=4n-d''
\end{equation}
\begin{equation}\label{fullranksecond4}
\rank\left(\JH^{\cs''}(\hat{\xx},\hat{\boldsymbol{\lambda}},\mathbf{0}),
I_{\rs''+3n}\right)=4n, \end{equation}
 where
$\JH^{\cs''}(\hat{\xx},\hat{\boldsymbol{\lambda}},\mathbf{0})$ is a
matrix obtained from
$\JH(\hat{\xx},\hat{\boldsymbol{\lambda}},\mathbf{0})$ by deleting
its $c_1'',c_2'',\ldots,c_{d''}''$-th columns, and
\begin{align}
I_{\rs''+3n}=\left(\begin{array}{c}
                                        \mathcal{O}_{3n,d''} \\
                                        I_{\rs''}
                                      \end{array}
\right),  ~ I_{\rs''}=( %consists of the  unit vectors
\ee_{k_1''},\ee_{k_2''},\ldots,\ee_{k_{d''}'}). \end{align}

\begin{theorem}\label{3rddfl}
Comparing to $\JG(\hat{\xx},\hat{\boldsymbol{\lambda}}_1,\mathbf{0})$, the corank of $\JH(\hat{\xx},\hat{\boldsymbol{\lambda}},\mathbf{0})$
% $\JG(\hat{\xx},\hat{\boldsymbol{\lambda}_1},\mathbf{0})$
does not increase, i.e.,  $d'' \leq d'$. Moreover, we can choose
$\cs''$ and $\rs''$ such that  $\cs'' \subseteq \cs'$,
$\rs''\subseteq\rs'$ and satisfy (\ref{fullranksecond3}) and (\ref{fullranksecond4}) respectively.%have the following relations
%\[~\cs'\subseteq\cs\mbox{ and }\rs'\subseteq\rs.\]
%The corank of the Jacobian matrix of $G(\xx,\boldsymbol{\lambda}_1,\bb_0)$
% $\JG(\hat{\xx},\hat{\boldsymbol{\lambda}_1},\mathbf{0})$
%does not increase, i.e.,
%$d'' \leq d'$, $\cs''
%\subseteq \cs'$ and
%$\rs''\subseteq\rs'$.%have the following relations
%\[~\cs'\subseteq\cs\mbox{ and }\rs'\subseteq\rs.\]
\end{theorem}
\begin{proof}
Similar to the proof of Theorem \ref{2nddfl},  let
$\JH^{\cs'}(\hat{\xx},\hat{\boldsymbol{\lambda}},\mathbf{0})$ be the
matrix obtained from
$\JH(\hat{\xx},\hat{\boldsymbol{\lambda}},\mathbf{0})$ by deleting
its $c'_1,c'_2,\ldots,c'_d$-th
  columns. By (\ref{fullranksecond1}) and (\ref{fullranksecond2}), we claim that
\begin{equation*}
\rank(\JH^{\cs'}(\hat{\xx},\hat{\boldsymbol{\lambda}},\mathbf{0}))=4n-d'.
\end{equation*}
Therefore, $d''\leq d'$, and there exists a positive-integer set
$\cs''\subseteq\cs'$ such that the condition (\ref{fullranksecond3})
is satisfied.
%. It is clear that  we can delete  the columns
%$\{j_1',j_2',\ldots,j_{d'}'\}$ among the columns
%$j_1,j_2,\ldots,j_d$ such that the obtained matrix
%$G_{{\xx},{\boldsymbol{\lambda}_1},\bb_0}^{\cs'}(\hat{\xx},\hat{\boldsymbol{\lambda}_1},\mathbf{0})$
%has full column rank. Hence $\cs'\subseteq\cs$.

Meanwhile, we know that
%\rank(\B,-I_{\rs})=
$\rank(\JG^{\cs'}(\hat{\xx},\hat{\boldsymbol{\lambda}}_1,\mathbf{0}),I_{\rs'+n})=2n$, then
\begin{equation*}
 \rank(\JH^{\cs'}(\hat{\xx},\hat{\boldsymbol{\lambda}},\mathbf{0}),I_{\rs'+3n})=4n,
\end{equation*}
%where $\C$ is obtained from $\JG(\hat{\xx},\hat{\boldsymbol{\lambda}_1},\mathbf{0})$ by deleting its
%$\cs$-th columns,
where $I_{\rs'+3n}=\left(\begin{array}{c}
                                        \mathcal{O}_{3n,d'} \\
                                        I_{\rs'}
                                      \end{array}
\right)$. Therefore, we can choose    $\rs''\subseteq\rs'$  such that the
condition (\ref{fullranksecond4}) is satisfied.
\end{proof}

%Inductively by Theorem \ref{2nddfl} and \ref{3rddfl}, the deflation method (\ref{Yamapara}) produces a decreasing nonnegative-integer sequence, whose $j$-th element
%is equal to the corank of the Jacobian matrix of the $j$-th order deflated system.

\begin{example}\label{ex1}\cite[DZ1]{DZ:2005}
Consider a polynomial system
\[F=\{x_1^4-x_2x_3x_4, x_2^4-x_1x_3x_4, x_3^4-x_1x_2x_4, x_4^4-x_1x_2x_3\}.\]
The system $F$ has $(0,0,0,0)$ as a $131$-fold isolated zero.
\end{example}
Since $\JF(\hat{\xx})=\mathcal{O}_{4,4}$, we derive $d=4$,
$\cs=\rs=\{1,2,3,4\}$ and $\vv_1=(1,1,1,1)^T$
\begin{align*}
G(\xx,\bb_0)=\left\{\begin{array}{r}
         F(\xx)-I_{\rs}\bb_0=\mathbf{0}, \\
         4x_1^3-x_3x_4-x_2x_4-x_2x_3=0, \\
         4x_2^3-x_3x_4-x_1x_4-x_1x_3=0, \\
         4x_3^3-x_2x_4-x_1x_4-x_1x_2=0, \\
         4x_4^3-x_2x_3-x_1x_3-x_1x_2=0.
       \end{array}
\right.
\end{align*}
The Jacobian matrix of $G(\xx,\bb_0)$ at $(\mathbf{0}, \mathbf{0})$ is
\begin{equation*}
G_{\xx,\bb_0}(\mathbf{0}, \mathbf{0})=\left(
    \begin{array}{cc}
      \mathcal{O}_{4,4} & -I_{\rs} \\
      \mathcal{O}_{4,4} & \mathcal{O}_{4,4} \\
    \end{array}
  \right),
\end{equation*}
 Hence, $d'=4$, $\cs'=\rs'=\{1,2,3,4\}$
and
\begin{align}\label{example}
H(\xx,\boldsymbol{\lambda},\bb)=\left\{\begin{array}{r}
         F(\xx)-I_{\rs}\bb_0-X_1\bb_1=\mathbf{0}, \\
         \JF(\xx)\vv_1-I_{\rs'}\bb_1=\mathbf{0},\\
         \widetilde{G}_{\xx,\bb_0}(\xx,\bb_0,\bb_1)\vv_2=\mathbf{0},
       \end{array}
\right.
\end{align}
where $\vv_2=(1,1,1,1,\lambda_1,\lambda_2,\lambda_3,\lambda_4)^T$,
and
$\widetilde{G}_{\xx,\bb_0}(\mathbf{0},\mathbf{0},\mathbf{0})\vv_2=\mathbf{0}$
has a unique solution $(\hat \lambda_1,\hat \lambda_2,\hat
\lambda_3,\hat \lambda_4)=(0,0,0,0)$. The Jacobian matrix of
$H(\xx,\boldsymbol{\lambda},\bb)$ at
$(\mathbf{0},\mathbf{0},\mathbf{0})$ is
\begin{equation*}
H_{\xx,\boldsymbol{\lambda},\bb}(\mathbf{0},\mathbf{0},\mathbf{0})=\left(
    \begin{array}{cccc}
      \mathcal{O}_{4,4} & -I_{\rs} & \mathcal{O}_{4,4} & \mathcal{O}_{4,4} \\
      \mathcal{O}_{4,4} & \mathcal{O}_{4,4} & \mathcal{O}_{4,4} & -I_{\rs'} \\
      \mathcal{O}_{4,4} & \mathcal{O}_{4,4} & -I_{\rs'} & -I_{\rs'} \\
      A & \mathcal{O}_{4,4} & \mathcal{O}_{4,4} & \mathcal{O}_{4,4}
    \end{array}
  \right),A=\left(
    \begin{array}{cccc}
      0 & -2 & -2 & -2 \\
      -2 & 0 & -2 & -2 \\
      -2 & -2 & 0 & -2 \\
      -2 & -2 & -2 & 0
    \end{array}
  \right).
\end{equation*}
The Jacobian matrix
$H_{\xx,\boldsymbol{\lambda},\bb}(\mathbf{0},\mathbf{0},\mathbf{0})$
is nonsingular. Therefore we obtain a regular and square system
$H(\xx,\boldsymbol{\lambda},\bb)$ and a perturbed system
\[\widetilde{F}(\xx,\bb)=\left\{\begin{array}{c}
                       x_1^4-x_2x_3x_4-b_1-b_5x_1 =0,\\
                       x_2^4-x_1x_3x_4-b_2-b_6x_2 =0,\\
                       x_3^4-x_1x_2x_4-b_3-b_7x_3 =0,\\
                       x_4^4-x_1x_2x_3-b_4-b_8x_4=0.
                     \end{array}
\right.\]

 Applying the verification method based on Theorem \ref{verification} to $H(\xx,\boldsymbol{\lambda},\bb)$,   we show  in Section \ref{mrver}  that a slightly
perturbed polynomial system $\widetilde{F}(\xx,\hat{\bb})$ for
\[|\hat{b}_i|\leq1.0e-321, i=1,2,\ldots,8\]
 has an isolated singular
solution $\hat{\xx}$ within \[|\hat{x}_i|\leq1.0e-321,
i=1,2,3,4.\]

%Suppose  $\tilde{\xx}=(.0003445,.0009502,.0003171,.0006948)$ and
%$\varepsilon=0.005$, we obtain the augmented system (\ref{example})
%and a point
%\[\tilde{\yy}=(\tilde{\xx},0.8009\times10^{-6},0.4236\times10^{-7},0.8859\times10^{-7},0.5374\times10^{-7},0,\ldots,0).\]
%After running \verb"verifynlss"$(H,\tilde{\yy})$ in Matlab
%\cite{RumpINT}, it yields
%\[-1.0e-321\leq \hat{x}_i\leq1.0e-321,\mbox{ for }i=1,2,3,4,\]
%\[-1.0e-321\leq \hat{b}_i\leq1.0e-321,\mbox{ for }i=1,2,\ldots,8.\]
%By Theorem \ref{breadth}, this proves that the perturbed polynomial
%system $\widetilde{F}(\xx,\hat{\bb})$ $(|\hat{b}_i|\leq1.0e-321,
%i=1,2,\ldots,8)$ has an isolated singular solution $\hat{\xx}$
%within $|\tilde{x}_i|\leq1.0e-321, i=1,2,3,4$.

\subsection{Higher-order deflations}

%Suppose the Jacobian matrix
%$\widetilde{H}_{\xx,\boldsymbol{\lambda},\bb}(\hat{\xx},\hat{\boldsymbol{\lambda}},\mathbf{0})$
%is still singular, then we  repeat the deflation.
For higher-order deflations, in the following, we show inductively
how to add  new smoothing parameters properly  to the original
system in order to derive a square and regular deflated system for
certifying the existence of an isolated singular solution of a
slightly perturbed system.

 Let $H^{(0)}(\xx)=F(\xx)$, then for %an isolated  singular solution $\hat{\xx}$ with high singularity,
 the $(s+1)$-th deflation, we add  smoothing parameters  $\bb^{(s)}=(\bb_0^T, \ldots, \bb_{s}^T)^T$ and consider the following square system
\begin{align}\label{aug}
H^{(s+1)}(\xx,\boldsymbol{\lambda}^{(s+1)},\bb^{(s)})=
\left\{\begin{array}{rl}
         \widetilde{F}(\xx,\bb^{(s)})&=\mathbf{0}, \\
         \widetilde{F}_{\xx}(\xx,\bb^{(s)})\vv_1&=\mathbf{0}, \\
         %\widetilde{G}_{\xx,\boldsymbol{\lambda}_1,\bb_0}(\xx,\boldsymbol{\lambda}_1,\bb)\vv_2=\mathbf{0}, \\
        % \cdots=\mathbf{0}.\\
        &\vdots\\
        %\end{array} \right.\\
 G^{(s)}_{\xx,\boldsymbol{\lambda}^{(s)}, \bb^{(s-1)}}
         (\xx,\boldsymbol{\lambda}^{(s)},\bb^{(s)})\vv_{s+1}&=\mathbf{0},
       \end{array}
\right.
\end{align}
where $\boldsymbol{\lambda}^{(s+1)}=(\boldsymbol{\lambda}_1^T,
\ldots, \boldsymbol{\lambda}_{s+1}^T)^T$ are extra variables
corresponding to the vectors $\{\vv_1, \ldots, \vv_{s+1}\}$,  $
G^{(s)}
(\xx,\boldsymbol{\lambda}^{(s)},\bb^{(s)})$ %=\mathbf{0}=\left\{\begin{array}{r}$
consists of the first $2^s n$ polynomials in
$H^{(s+1)}(\xx,\boldsymbol{\lambda}^{(s+1)},\bb^{(s)})$, and
\begin{equation}\label{perpoly}
  \widetilde{F}(\xx,\bb^{(s)})=F(\xx)-X_0\bb_0-X_1\bb_1-\cdots-X_s\bb_{s},
\end{equation}
 the matrix $X_j$ ($0 \leq j \leq s$) consists of vectors
$\frac{1}{j!}\cdot x^j_{\cs^{(j)}(i)}\cdot\ee_{\rs^{(j)}(i)}$,
$i=1,\ldots,d_j$, where  $\cs^{(j)}$
and $\rs^{(j)}$%+\{2^jn-n\}$
are two positive-integer sets selected at the $j$-th order deflation
satisfying conditions obtained by replacing the polynomial system
$F(\xx)$ in  (\ref{col}) and (\ref{row}) by the $j$-th deflated
system $H^{(j)}(\xx,\boldsymbol{\lambda}^{(j)},\bb^{(j-1)})$ and
replacing $I_\rs$ by the matrix
 $I_{\rs^{(j)}+ (2^j-1)n}=\left(\begin{array}{c}
                                        \mathcal{O}_{(2^j-1) n,d_j} \\
                                        I_{\rs^{(j)}}
                                      \end{array}
\right),  ~ I_{\rs^{(j)}}=( %consists of the  unit vectors
\ee_{k_1^{(j)}},\ee_{k_2^{(j)}},\ldots,\ee_{k_{d_j}^{(j)}})$,
 where   $d_j$  is the corank of
$H^{(j)}_{\xx,\boldsymbol{\lambda}^{(j)},\bb^{(j-1)}}(\hat{\xx},\hat{\boldsymbol{\lambda}}^{(j)},\mathbf{0})$. %The
%polynomial system $ G^{(s)}
%(\xx,\boldsymbol{\lambda}^{(s)},\bb^{(s)})$ %=\mathbf{0}=\left\{\begin{array}{r}$
%consists of the first $2^s n$ polynomials in
%$H^{(s+1)}(\xx,\boldsymbol{\lambda}^{(s+1)},\bb^{(s)})$.

%The augmented system  after $s+1$ deflations has the following form:
%\begin{align}\label{aug}
%H_{s+1}(\xx,\boldsymbol{\lambda},\bb)= \left\{\begin{array}{r}
%         \perF=\mathbf{0}, \\
%         \widetilde{F}_{\xx}(\xx,\bb)\vv_1=\mathbf{0}, \\
%         \widetilde{G}_{\xx,\boldsymbol{\lambda}_1,\bb_0}(\xx,\boldsymbol{\lambda}_1,\bb)\vv_2=\mathbf{0}, \\
%         \cdots=\mathbf{0}.
%       \end{array}
%\right.
%\end{align}

\begin{theorem} \label{maintheorem}
%Compared with %$H^{(s)}_{\xx,\boldsymbol{\lambda}^{(s)},\bb^{(s-1)}}(\xx,\boldsymbol{\lambda}^{(s)},\bb^{(s-1)})$,
%the $s$-th deflation,
The corank $d_{s+1}$ of
$H^{(s+1)}_{\xx,\boldsymbol{\lambda}^{(s+1)},\bb^{(s)}}(\hat{\xx},\hat{\boldsymbol{\lambda}}^{(s+1)},\mathbf{0})$
%of the $(s+1)$-th deflated system
does not increase  and the number of deflations needed to derive a
regular solution of an augmented system (\ref{aug}) is less than the
depth of $\mathcal{D}_{\hat{\xx}}$, i.e., we have
\begin{align}\label{corankseq}
%d^{(0)} \geq d^{(1)} \geq \cdots \geq d^{(s+1)}\geq \cdots \geq
%d^{(\rho-1)}=0.
d_0 \geq d_1 \geq \cdots \geq d_{s+1} \geq \cdots \geq d_{\rho-1}=0.
\end{align}
 Moreover, we can choose $\cs^{(j)}$ and $\rs^{(j)}$
satisfying
\begin{align}\label{cjkj}
\cs^{(s)} \subseteq \cdots \subseteq \cs^{(0)}\mbox{ and }
 \rs^{(s)} \subseteq \cdots \subseteq \rs^{(0)}.
 \end{align}
\end{theorem}
\begin{proof}
Applying Theorem \ref{2nddfl}, \ref{samerank} and \ref{3rddfl} inductively, we can show that
 the
above deflation process (\ref{aug}) produces a decreasing
nonnegative-integer sequence $d_0 \geq d_1 \geq \cdots \geq
d_{s+1}\geq \cdots $,
%we can remove all smoothing parameters to the original system inductively, to obtain a perturbed polynomial system
%whose $j$-th element is equal to the corank of the Jacobian matrix of the $j$-th order deflated system,
%As same as (\ref{Yamaper}) and (\ref{liper}), the Jacobian matrices of every order deflations share the same null space,
 which is as same as the sequence  consisting of  %lzhi rewrite
  coranks  of the Jacobian matrices of the deflated
 systems by Yamamoto's method.  According to   Theorem \ref{terminate}, the number of
Yamamoto's deflations  to derive a regular solution of an augmented
system is bounded by the depth of $\mathcal{D}_{\hat{\xx}}$. Hence
the number of the modified deflations (\ref{aug}) is also bounded by
the depth of $\mathcal{D}_{\hat{\xx}}$. %lzhi was: $\rho-1$.
The proof of
(\ref{cjkj}) is similar to the proofs of Theorem \ref{2nddfl} and
\ref{3rddfl}.
\end{proof}
\begin{theorem}\label{breadth}
Suppose Theorem \ref{verification} is applicable to the augmented
system (\ref{aug}), and yields inclusions for $\hat{\xx}$,
$\hat{\lam}$ and $\hat{\bb}$. Then the perturbed system
$\widetilde{F}(\xx,\hat{\bb})$ has an isolated singular solution at
$\hat{\xx}$.
%with $\rank(A(\hat{\xx}))=n-d$, where $A(\xx)$ is the Jacobian matrix of $F_0(\xx,\hat{\bb})$ and $1<d\leq n$.
\end{theorem}
\begin{proof}
Since $(\hat{\xx}, \hat{\lam}, \hat{\bb})$ is the unique solution of
the augmented system (\ref{aug}), we have
\begin{align*}
\widetilde{F}(\hat{\xx},\hat{\bb})=\mathbf{0} ~~{\text {and}}~~
\widetilde{F}_{\xx}(\hat{\xx},\hat{\bb})\hat{\vv}_1=\mathbf{0},
~\hat{\vv}_1\neq\mathbf{0}.
\end{align*}
 Hence, $\hat{\xx}$
is an isolated singular solution of the slightly perturbed system
\begin{align*}
\widetilde{F}(\xx,\hat{\bb})=F(\xx)-X_0 \hat{\bb}_0-X_1{\hat
{\bb}_1}-\cdots-X_s\hat{\bb}_s.
\end{align*}
\end{proof}
\begin{example}\label{ex2}\cite[DZ2]{DZ:2005}
Consider a polynomial system
\[F=\{x^4, x^2y+y^4, z+z^2-7x^3-8x^2\}.\]
The system $F$ has $(0,0,-1)$ as a $16$-fold isolated zero.
\end{example}
The Jacobian matrix of $F$ at $\hat{\xx}=(0,0,-1)$ is
\begin{equation*}
\JF(\hat{\xx})=\left(\begin{array}{ccc}
        0 & 0 & 0 \\
        0 & 0 & 0 \\
        0 & 0 & -1
      \end{array}
\right),\mbox{ so that }d_0=2\mbox{ and we
choose}~\cs^{(0)}=\rs^{(0)}=\{1,2\}.
\end{equation*}
The first-order deflated system is
\begin{align*}
H^{(1)}(\xx,\boldsymbol{\lambda}_1,\bb_0)=\left\{\begin{array}{r}
         F(\xx)-X_0 \bb_0=\mathbf{0}, \\
         4x^3=0, \\
         2xy+x^2+4y^3=0, \\
         -21x^2-16x+\lambda_1+2z\lambda_1=0,
       \end{array}
\right.
\end{align*}
where
\begin{align*}
X_0= (\ee_1, \ee_2)=\left(
\begin{array}{cc}
1 & 0\\
0&1\\
0 & 0
\end{array}
\right), ~~ \bb_0=\left(
\begin{array}{c}
b_1\\
b_2
\end{array}
\right), ~~\vv_1=(1,1,\lambda_1)^T,~~ \boldsymbol{\lambda}_1=(\lambda_1).
\end{align*}
The Jacobian matrix of $H^{(1)}(\xx,\boldsymbol{\lambda}_1,\bb_0)$
at $(0,0,-1,0,0,0)$ is
\begin{equation*}
%G_{\xx,\boldsymbol{\lambda}_1,\bb_0}(\hat{\xx},\mathbf{0})=
\left(
    \begin{array}{cccccc}
      0 & 0 & 0 & 0 & -1 & 0 \\
      0 & 0 & 0 & 0 & 0 & -1 \\
      0 & 0 & -1 & 0 & 0 & 0 \\
      0 & 0 & 0 & 0 & 0 & 0 \\
      0 & 0 & 0 & 0 & 0 & 0 \\
      -16 & 0 & 0 & -1 & 0 & 0
    \end{array}
  \right),~d_1=2\mbox{ and we choose}~\cs^{(1)}=\rs^{(1)}=\{1,2\}.
\end{equation*}
Therefore, we derive the second-order deflated system
\begin{align*}
H^{(2)}(\xx,\boldsymbol{\lambda}^{(2)},\bb^{(1)})=\left\{\begin{array}{r}
         F(\xx)-X_0\bb_0-X_1\bb_1=\mathbf{0}, \\
         \JF(\xx)\vv_1-X_1'\bb_1=\mathbf{0},\\
         {G}^{(1)}_{\xx,\boldsymbol{\lambda}_1,\bb_0}(\xx,\boldsymbol{\lambda}_1,\bb^{(1)})\vv_2=\mathbf{0},
       \end{array}
\right.
\end{align*}
where
\[X_1=\left(
\begin{array}{cc}
x & 0\\
0&y\\
0 & 0
\end{array}
\right), ~~ \bb_1=\left(
\begin{array}{c}
b_3\\
b_4
\end{array}
\right), ~~ \bb^{(1)}=(b_1,b_2,b_3,b_4)^T, ~~X_1'=\left(
\begin{array}{cc}
1 & 0\\
0&1\\
0 & 0
\end{array}
\right),\]
\[\vv_2=(1,1,\lambda_2,\lambda_3,\lambda_4,\lambda_5)^T,
~~ \boldsymbol{\lambda}^{(2)}=(\lambda_1,
\lambda_2,\lambda_3,\lambda_4,\lambda_5)^T.\]
 Moreover,
%$\widetilde{G}_{\xx,\boldsymbol{\lambda}_1,\bb_0}(\hat{\xx},\mathbf{0})\vv_2=\mathbf{0}$
 ${G}^{(1)}_{\xx,\boldsymbol{\lambda}_1,\bb_0}(\hat{\xx},\hat{\boldsymbol{\lambda}}^{(1)},\mathbf{0})\vv_2=\mathbf{0}$
has a unique solution  $\hat
{\boldsymbol{\lambda}}_2=(0,-16,0,0)^T$.

For the third-order deflation, we have $d_2=1$,
$\cs^{(2)}=\rs^{(2)}=\{1\}$, so
\begin{align}\label{regularex2}
%\left\{\begin{array}{r}
%         F(\xx)-X_0\bb_0-X_1\bb_1=\mathbf{0}, \\
%         \JF(\xx)\vv_1-X'_0\bb_1-X'_1\bb_2=\mathbf{0},\\
%         G_{\xx,\boldsymbol{\lambda}_1,\bb_0}(\xx,\boldsymbol{\lambda}_1,\bb)\vv_2=\mathbf{0},\\
%         H_{\xx,\boldsymbol{\lambda},\bb_0,\bb_1}(\xx,\boldsymbol{\lambda},\bb)\vv_3=\mathbf{0},
%       \end{array}
%\right.\Rightarrow
H^{(3)}(\xx,\boldsymbol{\lambda}^{(3)},\bb^{(2)})=
\left\{\begin{array}{r}
         F(\xx)-X_0\bb_0-X_1\bb_1-X_2\bb_2=\mathbf{0}, \\
         \JF(\xx)\vv_1-X_1'\bb_1-X'_2\bb_2=\mathbf{0},\\
         \JF(\xx)\vv_2'-X_0\vv_2''
         -X_1'\bb_1-X'_2\bb_2
         =\mathbf{0},\\
         F_{\xx\xx}(\xx)\vv_1\vv_2'+\JF^{\cs^{(0)}}(\xx)\lambda_3-X_2''\bb_2=\mathbf{0},\\
         %\widetilde{G}_{\xx,\boldsymbol{\lambda}_1,\bb_0}(\xx,\boldsymbol{\lambda}_1,\bb)\vv_2=\mathbf{0},\\
         {G}^{(2)}_{\xx,\boldsymbol{\lambda}^{(2)}, \bb^{(1)}}(\xx,\boldsymbol{\lambda}^{(2)},\bb^{(2)})\vv_3=\mathbf{0},
       \end{array}
\right.
\end{align}
where
\[X_2=\left(
\begin{array}{c}
\frac{1}{2}x^2 \\
0\\
0
\end{array}
\right), ~~ \bb_2= (b_5),~~X_2'=\left(
\begin{array}{c}
x \\
0\\
0
\end{array}
\right),~~X_2''=\left(
\begin{array}{c}
1 \\
0\\
0
\end{array}
\right),\]
\[\vv_2'=\left(\begin{array}{c}
                 1 \\
                 1 \\
                 \lambda_2
               \end{array}
\right),\vv_2''=\left(\begin{array}{c}
                 \lambda_4 \\
                 \lambda_5
               \end{array}
\right),\]
\[
 \vv_3=(1,\lambda_6,\lambda_7,\ldots,\lambda_{16})^T, ~~
 \boldsymbol{\lambda}^{(3)}=(\lambda_1,\ldots,
\lambda_{16})^T.\]
 Moreover,  ${G}^{(2)}_{\xx,\boldsymbol{\lambda}^{(2)},
\bb^{(1)}}(\hat{\xx},\hat{\boldsymbol{\lambda}}^{(2)},\mathbf{0})\vv_3=\mathbf{0}$
has a unique solution
\[\hat{\boldsymbol{\lambda}}_3=(-2,0,0,0,-16,0,0,-16,0,0,-42)^T.\]
%$\widetilde{H}_{\xx,\boldsymbol{\lambda}_1,\boldsymbol{\lambda}_2,\bb_0,\bb_1}(\hat{\xx},\hat{\boldsymbol{\lambda}},\mathbf{0})\vv_3=\mathbf{0}$.
Finally, the Jacobian matrix of
$H^{(3)}(\xx,\boldsymbol{\lambda}^{(3)},\bb^{(2)})$ is nonsingular,
and we obtain a perturbed polynomial system
\begin{align}\label{pertex2}
&\widetilde{F}(\xx,\bb)=F(\xx)-X_0\bb_0-X_1\bb_1-X_2\bb_2 \nonumber
\\ %=\left\{\begin{array}{r}
                      &=\{ x^4-b_1-b_3x-\frac{1}{2}b_5x^2,
                       x_2y+y^2-b_2-b_4y,
                       z+z^2-7x^3-8x^2\}.
                   %  \end{array}
%\right.
\end{align}
 Note that
 \[\JF(\xx)\vv_1-X_0'\bb_1-X'_1\bb_2=\mathbf{0}\Leftrightarrow\widetilde{F}_{\xx}(\xx,\bb)\vv_1=\mathbf{0},\]
after applying the verification method to the above regular
augmented system (\ref{regularex2}),  we are able to verify  that a
slightly perturbed system $\widetilde{F}(\xx,\hat{\bb})$ defined in
(\ref{pertex2}) for
\[|\hat{b}_i|\leq1.0e-14, ~~i=1,2,\ldots,5\]
 has an isolated singular
solution $\hat{\xx}$ within \[|\hat{x}_i|\leq1.0e-14, ~i=1,2,\mbox{
and} ~|1+\hat{x}_3|\leq1.0e-14.\]

\section{An Algorithm for Verifying Multiple Roots}\label{mrver}

%\begin{theorem}\label{termi}
%The number of deflations needed to derive a regular solution of a square augmented system is less than the breadth of the isolated solution.
%\end{theorem}
%The definition of \emph{breadth} was introduced in \cite{DZ:2005}, and the proof is very similar to the proof of \cite[Theorem 3]{DZ:2005} using \emph{local dual space}.
%Theorem \ref{termi} guarantees the termination of Algorithm \ref{MDSP}.
%In Section \ref{dfl},   we show  how to obtain a regular and square
%augmented system  (\ref{aug}) by adding smoothing parameters only to
%the original system in order to certify the existence of an isolated
%singular solution.
 Based on Theorem \ref{maintheorem} and
\ref{breadth}, we propose below an  algorithm for computing verified error
bounds such that, a slightly perturbed system is guaranteed to
possess an isolated singular solution within the computed bounds.

\begin{algorithm}\label{MDSP}VISS

\noindent \textbf{Input:} A square  polynomial system $F\in \CC[x_1,
\ldots, x_n]$, a point $\tilde{\xx}\in\mathbb{C}^n$ and a tolerance
$\varepsilon$.

\vskip 0.1cm
 \noindent \textbf{Output:} A perturbed
system $\widetilde{F}(\xx, \bb)$, inclusions $\mathbf{X}$ and
$\mathbf{B}$ for $\hat{\xx}$ and $\hat{\bb}$ such that
$\widetilde{F}(\hat \xx,\hat \bb)=\mathbf{0}$ and $\widetilde{F}_{\xx}(\hat \xx, \hat
\bb)$ is singular.
%\xx$, where $\hat{\bb}$ is a perturbation lying in the
 %interval domain  $B$.
%An augmented system $G$ and a point
%$\tilde{\yy}$.
\begin{enumerate}

\item  Set $s:=0$, $m:=n$, $\widetilde{F}:=F$, $G:=\widetilde{F}$,
$\yy:=\xx$, and $\tilde{\yy}:=\tilde{\xx}$.

  \item Compute $d:=n-\rank(\JF(\tilde{\xx}),\varepsilon)$,
   select integer sets $\cs$ and $\rs$
 %$\rs$
 satisfying (\ref{col}) and (\ref{row}) respectively.

  \item Set
  $\widetilde{F}:=\widetilde{F}+X_s\bb_s$,
   where the matrix $X_s$ consists of vectors $\frac{1}{s!}\cdot x^s_{\cs(i)}\cdot\ee_{\rs(i)}$, $i=1,\ldots,d$.

      \begin{enumerate}

        \item  If $s \geq 1$, then set $G:=\widetilde{F}$; for $j$ from 1 to s do\\
             $G:=\{G, G_{\yy}\vv_{j}\}$;
             $\yy:=(\yy,\boldsymbol{\lambda}_{j},\bb_{j-1}).$
           %\begin{enumerate}
%           \item $G:=\widetilde{F}$;
%            \item  $G:=\{G, G_{\yy}\vv_{j}\}$;
%            \item  $\yy:=(\yy,\boldsymbol{\lambda}_{j},\bb_{j-1}).$
      %Set $G:=\{G, G_{\yy}\vv_{s+1}\}$ and
     % Update $G$ then $G:=\{G, G_{\yy}\vv_{s+1}\}$, and
%      $\yy:=(\yy,\boldsymbol{\lambda}_{s+1},\bb_{s})$;
         %\end{enumerate}

      \item
  Compute $\tilde{\yy}:=(\tilde{\yy},\mathrm{LeastSquares}(G_{\yy}(\tilde{\yy})
  \vv_{s+1}=\mathbf{0}),\mathbf{0})$;

  \item Set $G:=\{G, G_{\yy}\vv_{s+1}\}$; $\yy:=(\yy,\boldsymbol{\lambda}_{s+1},\bb_{s})$;  $m:=2m$.

 % \item  Set $G:=\{G, G_{\yy}\vv_{s+1}\}$, $\yy:=(\yy,\boldsymbol{\lambda}_{s+1},\bb_{s})$ and $m:=2m$.

 \end{enumerate}

  \item Compute $d:=m-\rank(G_{\yy}(\tilde{\yy}),\varepsilon)$;

  \begin{enumerate}

\item  If  $d=0$, apply \verb"verifynlss"  to   $G$ and $\tilde{\yy}$ to compute inclusions $\mathbf{X}$ and $\mathbf{B}$ for $\hat{\xx}$ and $\hat{\bb}$.
% return $(G,\tilde{\yy})$;

\item  Otherwise, select $\cs$, $\rs$ satisfying (\ref{col}),(\ref{row}) for the polynomial system $G$,
  set $s:=s+1$, $\yy=\xx$ and go back to Step 3.

\end{enumerate}

\end{enumerate}
\end{algorithm}

\paragraph{Example 3.1}(continued) Given an approximate singular solution
$\tilde{\xx}=(.0003445,.0009502,.0003171,.0006948)$ and a tolerance
$\varepsilon=0.005$, we obtain the augmented system (\ref{example})
and a point
\[\tilde{\yy}=(\tilde{\xx},0.8009\times10^{-6},0.4236\times10^{-7},0.8859\times10^{-7},0.5374\times10^{-7},0,\ldots,0).\]
After running \verb"verifynlss"$(H,\tilde{\yy})$ in Matlab
\cite{RumpINT}, it yields
\[-1.0e-321\leq \hat{x}_i\leq1.0e-321,\mbox{ for }i=1,2,3,4,\]
\[-1.0e-321\leq \hat{b}_i\leq1.0e-321,\mbox{ for }i=1,2,\ldots,8.\]
By Theorem \ref{breadth}, this proves that the perturbed polynomial
system $\widetilde{F}(\xx,\hat{\bb})$ $(|\hat{b}_i|\leq1.0e-321,
i=1,2,\ldots,8)$ has an isolated singular solution $\hat{\xx}$
within $|\tilde{x}_i|\leq1.0e-321, i=1,2,3,4$.

\paragraph{Special case}
The breadth-one case where the corank of the Jacobian matrix equals one occurs frequently, and can be treated more efficiently.

%Suppose the corank of the Jacobian matrix $F_{\xx}(\hat \xx)$ is one
%(breadth-one case), % then we can compute the verified error bounds
%more efficiently.
In fact, we have shown in \cite[Theorem 3.8]{LZ:2011} that each step
of deflation described by (\ref{modifieddfl})  only reduces the
multiplicity $\mu$ of the singular solution  $\hat \xx$   by $1$.
According to Theorem  \ref{maintheorem}, the number of deflations
described by (\ref{aug}) will be $\mu-1$. Hence, Algorithm
\verb"VISS" generates an augmented regular system of the size
$(2^{\mu-1}n)\times(2^{\mu-1}n)$.  However, in \cite{LZ:2012}, we
introduced a more  efficient method based on the parameterized
multiplicity structure, to obtain a deflated regular system
$G(\xx,\bb,\boldsymbol{\lambda})$ which is of the size $(\mu
n)\times(\mu n)$ and can be used to  verify not only  the existence
of an isolated singular solution,   but also its   multiplicity
structure.

Let us  introduce briefly the method in \cite{LZ:2012} for the
special case of breadth one.  By adding $\mu-1$ smoothing parameter
$b_0,b_1,\ldots,b_{\mu-2}$ to a well selected polynomial, assumed to
be $f_1$, %lzhi add
we derive
a square augmented system
\[G(\xx,\bb,\boldsymbol{\lambda})=\left(\begin{array}{c}
                                        \widetilde{F}(\xx,\bb) \\
                                        L_1(\widetilde{F})\\
                                        \vdots\\
                                        L_{\mu-1}(\widetilde{F})
                                      \end{array}
\right)=\mathbf{0},
\mbox{ where }
\widetilde{F}(\xx,\bb)=\left(\begin{array}{c}
                               f_{1}(\xx)-\sum_{\nu=0}^{\mu-2}\frac{b_{\nu}x_1^{\nu}}{\nu!} \\
                               f_{2}(\xx) \\
                               \vdots \\
                               f_{n}(\xx)
                             \end{array}
\right),\] and $L_1,\ldots,L_{\mu-1}$ are parameterized bases
of the local dual space in variables $\boldsymbol{\lambda}$.
Furthermore, we proved that if Theorem \ref{verification} is
applicable to $G$ and yields inclusions for
$\hat{\xx}\in\mathbb{R}^n$, $\hat{\bb}\in\mathbb{R}^{\mu-1}$ and
$\hat{\boldsymbol{\lambda}}\in\mathbb{R}^{(\mu-1)\times(n-1)}$ such
that $G(\hat{\xx},\hat{\bb},\hat{\boldsymbol{\lambda}})=\mathbf{0}$,
then $\hat{\xx}$ is a breadth-one singular solution of
$\widetilde{F}(\xx,\hat{\bb})=\mathbf{0}$ with multiplicity $\mu$
and $\{1,L_1,\ldots,L_{\mu-1}\}$ with
$\boldsymbol{\lambda}=\hat{\boldsymbol{\lambda}}$ is a basis of
$\mathcal{D}_{\hat{\xx}}$.
%
%Notice that $G(\xx,\bb,\boldsymbol{\lambda})$ is of the size $(\mu n)\times(\mu n)$ while Algorithm MDSP generates an augmented system of the size $(2^{\mu-1}n)\times(2^{\mu-1}n)$.
%Another thing should be noted is that, in general cases we prove the existence of an isolated singular solution but with no certification of the multiplicity; in the breadth-one case, we verify the multiplicity as well as the multiplicity strucutre.
%In \cite{LZ:2012}, we proposed a method to add $n$ new equations in one deflation instead of doubling the size. Moreover, we not only verified an isolated breadth-one solution of a slightly perturbed system, but also the multiplicity structure.
%\begin{remark}
%The case where $\rank(A(\tilde{\xx}))=n-1$ (breadth one) occurs frequently and can be treated more efficiently than the general case. In \cite{LZ:2012}, we proposed a method to add $n$ new equations in one deflation instead of double the size. Moreover, we not only verified an isolated breadth-one solution of a slightly perturbed system, but also the multiplicity.
%\end{remark}
\begin{example}\cite[Example 4.11]{RuGr09}
Consider a polynomial system
\[F=\{x_1^2x_2-x_1x_2^2,x_1-x_2^2\}.\]
The system $F$ has $(0,0)$ as a $4$-fold isolated zero.
\end{example}
%The Jacobian matrix of $F$ at $(0,0)$ is
%\[\JF(0,0)=\left[\begin{array}{cc}
%  0 & 0  \\
%  1 & 0
%\end{array}
%\right].\]
We choose $x_2$ as the perturbed variable and  add the
 univariate polynomial $-b_1-b_2x_2-\frac{b_3}{2}x_2^2$
 to the first equation in $F$ to obtain an augmented
system
\[\left\{\begin{array}{r}
            x_1^2x_2-x_1x_2^2-b_1-b_2x_2-\frac{b_3}{2}x_2^2=0, \\
            x_1-x_2^2=0, \\
            2\lambda_1x_1x_2-\lambda_1x_2^2+x_1^2-2x_1x_2-b_2-b_3x_2=0, \\
            \lambda_1-2x_2=0, \\
            \lambda_1^2x_2+2\lambda_1x_1-2\lambda_1x_2+2\lambda_2x_1x_2-\lambda_2x_2^2-x_1-\frac{b_3}{2}=0, \\
            \lambda_2-1=0, \\
            \lambda_1^2+2\lambda_1\lambda_2x_2-\lambda_1+2\lambda_2x_1-2\lambda_2x_2+2\lambda_3x_1x_2-\lambda_3x_2^2=0, \\
            \lambda_3=0,
          \end{array}
\right.\]
which is of the size $8\times8$ while Algorithm \verb"VISS" generates a system of the size $16\times16$.
Applying \verb"verifynlss" with an initial approximation
\[(0.002,0.003,-0.001,0.0015,-0.002,0.002,1.001,-0.01),\]
we obtain inclusions
\[-1.0e-14\leq \hat{x}_i\leq1.0e-14,\mbox{ for }i=1,2,3,\]
\[-1.0e-14\leq \hat{b}_i\leq1.0e-14,\mbox{ for }i=1,2,3.\]
This proves that the perturbed system $\widetilde{F}(\xx,\hat{\bb})$
($|\hat{b}_i| \leq 10^{-14}, i=1,2,3$) has a $4$-fold breadth-one
root $\hat{\xx}$ within $|\hat{x}_i| \leq 10^{-14}, i=1,2,3$.

\section{Experiments}\label{exp}
%\begin{example}
%Consider a slightly perturbed polynomial system from DZ1
%\[F=\{x_1^4-x_2x_3x_4+10^{-12}, x_2^4-x_1x_3x_4, x_3^4-x_1x_2x_4, x_4^4-x_1x_2x_3-10^{-12}x_4\}.\]
%%The system $F$ has $(0,0,0,0)$ as a $131$-fold isolated zero.
%\end{example}
%The ``approximate'' polynomial system $F$ has no isolated singular solutions, but a ``cluster'' of roots near the original point. We still can obtain the same perturbed polynomial system and the same augmented system by Algorithm \verb"MDSP".
%Applying \verb"verifynlss", it yields
%\[-1.0e-25\leq \hat{x}_i\leq1.0e-25,\mbox{ for }i=1,2,3,4,\]
%\[10^{-12}-1.0e-25\leq \hat{b}_1\leq10^{-12}+1.0e-25,\]
%\[-1.0e-25\leq \hat{b}_i\leq1.0e-25,\mbox{ for }i=2,3,\ldots,7,\]
%\[-10^{-12}-1.0e-25\leq \hat{b}_8\leq-10^{-12}+1.0e-25.\]
%By Theorem \ref{breadth}, this proves that the perturbed polynomial
%system $\widetilde{F}(\xx,\hat{\bb})$ has an isolated singular
%solution $\hat{\xx}$ within the above computed error bounds.

We can generate an augmented square and regular system  and initial
values for $\tilde \yy$ in Maple or Matlab, then
 apply INTLAB
function \textsf{verifynlss} in Matlab \citep{RumpINT} to obtain the
verified error bounds.  The following experiments are done in
Maple~15 for $\text{Digits}:=14$ and Matlab~R2011a with INTLAB\_V6
under Windows~7. Let  $n$  be the number of polynomials and
variables, $\mu$ be the multiplicity. The fourth and fifth  column
show the decrease of the corank and the increase of the smallest
singular values  of the Jacobian matrix respectively. The last two
columns give qualities of the verified error bounds.

The first three examples DZ1, DZ2, DZ3 are cited from \cite{DZ:2005}. It should
be noticed that the coefficients of  polynomials in the example DZ3
 have  algebraic
numbers $\sqrt{5}, \sqrt{7}$. These irrational coefficients are
rounded  to fourteen digits in Maple or Matlab.  The other
examples are quoted from the PHCpack
demos by Jan Verschelde. %({http://homepages.math.uic.edu/~jan/}).
Codes of  Algorithm VISS and examples are available at {\small
\url{http://www.mmrc.iss.ac.cn/~lzhi/Research/hybrid/VISS}}.

\begin{table*}[ht]
\caption{\label{table1} Algorithm Performance}
%\begin{table}[ht]\label{table1}
\begin{center}
\begin{tabular}{|l|c|c|l|l|l|l|} \hline
System & $n$ & $\mu$ & $\corank(G_{\yy}(\tilde \yy))$ & Smallest
$\sigma$ & $\|\mathbf{X}\|$ & $\|\mathbf{B}\|$\\ \hline DZ1 & 4 &
131 & 4 $\rightarrow$ 4
$\rightarrow$ 0 & 1.1e-07 $\rightarrow$ 6.2e-01 & e-321 & e-321\\
\hline DZ2 & 3 & 16 & 2 $\rightarrow$ 2 $\rightarrow$ 1
$\rightarrow$ 0 & 7.1e-11 $\rightarrow$ 5.3e-03 & e-14 & e-14\\
\hline
DZ3 & 2 & 4 & 1 $\rightarrow$ 1 $\rightarrow$ 1 $\rightarrow$
0 &
2.2e-04 $\rightarrow$ 9.6e-03 & e-7 & e-7\\
 \hline cbms1 & 3
& 11 & 3 $\rightarrow$ 0 & 5.5e-04 $\rightarrow$ 1.0e-00 & e-321 &
e-321\\ \hline cbms2 & 3 & 8 & 3
$\rightarrow$ 0 & 3.2e-04 $\rightarrow$ 1.0e-00 & e-321 & e-321\\
\hline mth191 & 3 & 4 & 2 $\rightarrow$ 0 & 2.5e-04 $\rightarrow$
3.7e-01 & e-14 & e-14\\ \hline KSS & 10 & 638 & 9 $\rightarrow$ 0 &
6.5e-05 $\rightarrow$ 3.0e-01 & e-14 & e-14\\ \hline Caprasse & 4 &
4 & 2 $\rightarrow$ 0 & 1.4e-03 $\rightarrow$ 9.9e-01 & e-14 &
e-14\\ \hline
cyclic9 & 9 &
4 & 2 $\rightarrow$ 0 & 2.1e-10 $\rightarrow$ 3.8e-01 & e-13 &
e-13\\ \hline
RuGr09 & 2 &
4 & 1 $\rightarrow$ 1 $\rightarrow$ 1 $\rightarrow$ 0 & 3.0e-07 $\rightarrow$ 1.0e-00 & e-14 &
e-14\\ \hline
LiZhi12 & 100 &
3 & 1 $\rightarrow$ 1 $\rightarrow$ 0 & 3.6e-12 $\rightarrow$ 2.2e-05 & e-14 &
e-14\\ \hline
Ojika1 & 2 &
3 & 1 $\rightarrow$ 1 $\rightarrow$ 0 & 3.7e-04 $\rightarrow$ 5.6e-02 & e-14 &
e-14\\ \hline
Ojika2 & 3 &
2 & 1 $\rightarrow$ 0 & 9.9e-04 $\rightarrow$ 4.6e-01 & e-14 &
e-14\\ \hline
Ojika3 & 3 &
2 & 1 $\rightarrow$ 0 & 9.6e-05 $\rightarrow$ 5.0e-02 & e-14 &
e-14\\ \hline
Ojika4 & 3 &
3 & 1 $\rightarrow$ 1 $\rightarrow$ 0 & 1.2e-04 $\rightarrow$ 2.0e-00 & e-14 &
e-14\\ \hline
Decker2 & 3 &
4 & 1 $\rightarrow$ 1 $\rightarrow$ 1 $\rightarrow$ 0 & 2.2e-09 $\rightarrow$ 1.0e-00 & e-14 &
e-14\\ \hline

\end{tabular}
%\vspace{0.2cm}
%
% {\label{table1}Table: Algorithm Performance}
\end{center}
%\vspace{-0.4cm}
\end{table*}

\paragraph{Acknowledgements}
%The first author is grateful to Anton Leykin for fruitful
%discussions during the IMA summer program at Georgia Tech, 2012.

The authors are grateful to Yijun Zhu for helping us implement Algorithm
VISS in  Matlab. The first author is grateful to Anton Leykin for
fruitful discussions during the IMA summer program at Georgia Tech,
2012.

This research is supported  by NKBRPC 2011CB302400 and the Chinese
National Natural Science Foundation under Grants: 91118001,
60821002/F02, 60911130369 and 10871194.

\bibliographystyle{siam}
\bibliography{strings,wuxiaoli,linan,zhi}
\end{document}